\documentclass[11pt,DIV12,a4paper]{scrartcl}
\usepackage[utf8]{inputenc}
\usepackage[english]{babel}
\usepackage[babel]{csquotes}
\usepackage[backend=bibtex, hyperref]{biblatex}
\bibliography{AlmostSure}

\usepackage{amsmath}
\usepackage{amsfonts}
\usepackage{amssymb}
\usepackage{amsthm} 
\usepackage{dsfont} 
\usepackage[mathscr]{euscript} 
\usepackage{eurosym}
\usepackage{bbm} 

\usepackage{comment} 
\usepackage{paralist} 

\usepackage{graphicx}
\usepackage{xkeyval}
\usepackage{pstricks}
\usepackage{hyperref}

\DeclareMathOperator{\CovMat}{CovMat}

\newcommand{\Cat}{\mathscr{C}}
\newcommand{\Prob}{\mathds{P}}

\newcommand{\E}{\mathds{E}}
\newcommand{\V}{\mathds{V}}

\newcommand{\R}{\mathbb{R}}

\newcommand{\N}{\mathbb{N}}

\newcommand{\mc}[1]{\mathcal{#1}}
\newcommand{\abs}[1]{|{#1}|} 
\newcommand{\bigabs}[1]{\left|{#1}\right|} 
\newcommand{\one}{\mathds{1}}

\newcommand{\boundedcont}[1]{\mathcal{C}_{b}({#1})}

\newcommand*{\defeq}{\mathrel{\vcenter{\baselineskip0.5ex \lineskiplimit0pt
                     \hbox{\scriptsize.}\hbox{\scriptsize.}}}%
                     =}
										
\newcommand{\integrala}[2]{\left\langle{#1},{#2}\right\rangle}

\newcommand{\integrald}[4]{\int_{#4}{#3}\,{#2(\mathrm{d}{#1})}}

\newcommand{\Acal}{\mathcal{A}}

\newcommand{\Bcal}{\mathcal{B}}

\newcommand{\norm}[1]{\|#1\|}

\newcommand{\ubar}[1]{\underline{#1}}
\newcommand{\oneto}[1]{[{#1}]}
\newcommand{\sqr}[1]{\square_{#1}}

\DeclareMathOperator{\tr}{tr}

\newcommand{\PePe}[1]{\mathcal{PP}(#1)}

\theoremstyle{plain}
\newtheorem{lemma}{Lemma}[section]
\newtheorem{theorem}[lemma]{Theorem}

\newtheorem{corollary}[lemma]{Corollary}
\theoremstyle{definition}
\newtheorem{definition}[lemma]{Definition}

\newtheorem{example}[lemma]{Example}
\newtheorem{remark}[lemma]{Remark}
\theoremstyle{remark}

\begin{document}
\begin{center}
{\LARGE \textbf{\textsf{The Almost Sure Semicircle Law for Random Band\\
\vspace{0.2cm}
Matrices with Dependent Entries}}}\\
\vspace{0.8cm}
{\Large Michael Fleermann, Werner Kirsch and Thomas Kriecherbauer}
\end{center}
\vspace{0.8cm}

\begin{abstract}
\textbf{Abstract.}We analyze the empirical spectral distribution of random periodic band matrices with correlated entries.  The correlation structure we study was first introduced in 2015 in \autocite{HKW} by Hochstättler, Kirsch and Warzel, who named their setup \emph{almost uncorrelated} and showed convergence to the semicircle distribution in probability. We strengthen their results which turn out to be also valid almost surely. Moreover, we extend them to band matrices. Sufficient conditions for convergence to the semicircle law both in probability and almost surely are provided. In contrast to convergence in probability, almost sure convergence seems to require a minimal growth rate for the bandwidth. Examples that fit our general setup include Curie-Weiss distributed, correlated Gaussian, and as a special case, independent entries.
\end{abstract}


\section{Introduction}

The theory of random matrices had its origins in the applications, namely in statistics, where John Wishart analyzed properties of multivariate normal populations (see \autocite{Wishart}), and mathematical physics, where Eugene Wigner studied energy levels of heavy nuclei (see \autocite{Wigner}). Ever since, the reach of this theory has grown tremendously, with fruitful interactions with the fields of information theory (e.g. wireless communication, see \autocite{Telatar}), biology (e.g. RNA analysis, see \autocite{RNA}) and with various branches of mathematics (e.g. free probability, see \autocite{Speicher}).

In the present paper we restrict our attention to the real symmetric case. By a \emph{random matrix} we understand a symmetric $n\times n$ matrix $X_n$, whose entries $X_n(i,j)$ are real-valued random variables on some probability space $(\Omega,\Acal,\Prob)$. Due to symmetry, $X_n$ possesses $n$ real eigenvalues $\lambda^{X_n}_1\leq\ldots\leq\lambda^{X_n}_n$. Based on these random eigenvalues, we consider the empirical spectral distribution (ESD)
\[
\sigma_n \defeq \frac{1}{n} \sum_{i=1}^n \delta_{\lambda_i^{X_n}},
\]
where for each $a\in\R$ we denote by $\delta_a$ the Dirac measure in $a$.
Given such a sequence $(X_n)_n$ we obtain in this way a sequence $(\sigma_n)_n$ of random probability measures, and we can analyze its weak convergence in some probabilistic sense. Arguably the most famous result in this direction is \emph{Wigner's semicircle law} for ensembles with entries which are \emph{independent} (up to the symmetry constraint) and satisfy  $\E X_n(i,j)=0$ and $\E X^2_n(i,j)=n^{-1}$.  It states under mild

\vspace{0.5cm}
\footnoterule
{\footnotesize
\noindent\textit{Date:} October 22, 2019 \newline
\textit{2010 Mathematics Subject Classification:} 60B20\newline
\textit{Keywords:} random matrix, band matrix, dependent entries, semicircle law}

\noindent
 additional conditions that the sequence $\sigma_n$ converges weakly almost surely to the semicircle distribution $\sigma$ on the real line given by its Lebesgue density $\frac{1}{2\pi}\sqrt{4-x^2}\one_{[-2,2]}(x)$ (see e.g.\ the monographs \autocite{Anderson, BaiSi, Tao}). That is, we find a set $A\in\Acal$ with $\Prob(A)=1$, such that for all $\omega\in A$ we have that $\sigma_n(\omega)\to \sigma$ weakly as $n\to\infty$.

Wigner's semicircle law can be viewed as a central limit theorem in random matrix theory. In classical probability theory, the central limit theorem holds even if random variables are mildly correlated. Therefore, in the context of random matrices, a natural question to ask is whether one can relax the assumption of independence in Wigner's semicircle law and still obtain the semicircle distribution as a limit distribution of the ESDs. Another important way to deviate from the classical analysis is to consider band matrices instead of full matrices. In this paper, we derive semicircle laws for random band matrices (including full matrices as a special case) with certain correlated entries. The correlation structure we study is based on a model presented in \autocite{HKW}: In their paper, Hochst\"attler, Kirsch and Warzel defined a new scheme of random matrices they called \emph{approximately uncorrelated}. The entries in these matrices are allowed to exhibit a correlation structure which includes in particular Curie-Weiss models. They showed convergence of the empirical spectral distribution to the semicircle law weakly in probability. In this paper, we address the question of almost sure convergence. To this end, we need to refine the combinatorics of \autocite{HKW}. This allows us to identify additional conditions -- in particular a fourth moment condition -- that guarantee almost sure convergence. Our ensembles are still general enough to include Curie-Weiss models. In addition, our results also apply to random matrices with certain correlated Gaussian entries. Moreover, our method carries over to associated ensembles of periodic band matrices (Theorem~\ref{thm:main} 2,3). Observe that we need to require a minimal growth condition on the bandwidths to obtain almost sure convergence. Interestingly, for convergence in probability it suffices that the bandwidths grow to infinity without any assumptions on the rate (Theorem~\ref{thm:main} 1).
 Finally we would like to point out that even in the case of independent entries, our almost sure result on periodic band matrices (Theorem~\ref{thm:main} 4) appears to be new (cf.\ Problem 2.4.13 in \autocite{PasturBook} for the i.i.d.\ Gaussian case).

Matrix ensembles with correlated entries have been studied, for example, in the publications \autocite{Schenker, Friesen:eins, Friesen:zwei}, where semicircle laws in probability or almost surely were derived. In \autocite{Schenker}, semicircle laws were derived in expectation for random matrices with sparse dependencies. In \autocite{Friesen:eins} and \autocite{Friesen:zwei}, semicircle laws were derived for ensembles with independent diagonals. In contrast, we allow all entries to be correlated.

Recently, results on local laws (which are much stronger than almost sure limit laws) for random matrices with certain correlated entries were obtained in \autocite{AjankiErdos2016,ErdosKruger2017,AjankiErdos2018,ZiliangChe2017}.
Those authors require decay of correlations in some notion of distance of the entries, where the decay goes to zero for large distances, enabled by large matrix dimensions $n$. 
In our models, correlations are independent of the location of the entries, but depend only on the dimension $n$ of the matrix. This allows us to require much less decay in $n$ than in above mentioned publications. In our setup, all correlations between any two distinct entries may be of order $n^{-1}$. In contrast to this, in \autocite{AjankiErdos2016}, where Gaussian ensembles were studied, correlations need to decay at an exponential rate in the distance of the entries, see their condition (2.5). In \autocite{ErdosKruger2017}, which improves upon (an earlier preprint version of) \autocite{AjankiErdos2018}, a higher polynomial decay of the correlations are required, see their assumption (D). Lastly, \autocite{ZiliangChe2017} requires independence of entries far apart, see their Definition 2.1.

Previous results pertaining to periodic band matrices are contained in \autocite{Bogachev}, where the semicircle law was shown in probability for independent entries, and \autocite{Catalano2016}, where in addition, sparse dependencies were studied. We improve upon their results in two ways: On the one hand, we establish the almost sure semicircle law for Wigner type random periodic band matrices, where we require $(n b_n)^{-1}$ to be summable. On the other hand, we allow all entries to be correlated.

Another contribution of this paper are results concerning random matrices with correlated Gaussian entries. Previous work in this direction includes  the work of \autocite{MonvelKhorunzhy1999}, \autocite{BannaMerlevede2015} and \autocite{ChakrabartyHazra2016}, where various correlation structures, different from ours, are assumed. The only requirement we impose on the covariance matrices is that off-diagonal elements decay at a uniform rate. As Gaussian entries are an example fitting our general model, we obtain semicircle laws for periodic random band matrices (including the case of full matrices) in probability and almost surely, where for the latter results we require a minimal growth rate of the bandwidth.

This paper, which is a revised excerpt of the dissertation \autocite{FleermannDiss} of the first author, is organized as follows: In the second section we state our main result, Theorem~\ref{thm:main}, which covers many different cases. This becomes clear in the third section, where we discuss various examples, including Curie-Weiss distributed and Gaussian entries. We also elaborate on the connection to the previous almost uncorrelated scheme developed in \autocite{HKW}. The fourth and final section is devoted to the proof of the main theorem.

We end this introduction with a remark on possible extensions of our results that we plan to pursue in future work. The Curie-Weiss models discussed in the present paper all belong to the cases of critical or super-critical temperatures ($\beta \leq 1$) for which the correlations between distinct entries  decay to $0$ as the matrix size tends to infinity. We believe that our method can also be used to investigate sub-critical temperatures ($\beta > 1$) where the correlations neither vanish nor become sparse in the limit $n \to \infty$. In addition, we want to extend our results from periodic band matrices, where each row has the same number of identically vanishing entries, to strict band matrices where this property is violated for a (not necessarily small) number of rows that are near the top or near the bottom of the matrices.

\section{Statement of Main Results}

To state the main theorem, we need to introduce some concepts and notation.
For all $n\in\N$ we denote by $\oneto{n}$ the set of indices $\{1,\dots,n\}$, and by $\sqr{n}$ the set of index pairs $\{1,\ldots,n\}\times \{1,\ldots,n\}$. Let $(\Omega,\Acal,\Prob)$ be a probability space, and for all $n\in\N$ let $\{a_n(p,q): (p,q) \in \sqr{n} \}$ be a family of real-valued random variables. Then we call the sequence $(a_n)_{n\in\N}$ a \emph{triangular scheme}, if each element $a_n$ is a symmetric $n\times n$ random matrix. In particular, $a_n$ depends on $\{p,q\}$ rather than on $(p, q)$. We therefore
call two pairs of numbers $(a,b),~(c,d)\in\N^2$ \emph{fundamentally different} if $\{a,b\}\neq \{c,d\}$.

We now define $\alpha$-almost uncorrelated triangular schemes. These are the non-scaled versions of the random matrices that we study in this paper.

\begin{definition}
\label{def:alphaalmostuncorrelated}
Let $(a_n)_{n\in\N}$ be a triangular scheme and $\alpha>0$. Then we call $(a_n)_{n\in\N}$ \emph{$\alpha$-almost uncorrelated}, if 
there exist sequences of constants $(C_k)_k$, $(C^{(m)}_n)_{n,m}$, $(D^{(m)}_n)_{n,m}$ with $\lim_{n\to \infty} C^{(m)}_n=0$, $\lim_{n\to \infty} D^{(m)}_n=0$ for each $m \in \N$, such that for all
$N,l\in\N$ and fundamentally different pairs $(p_1,q_1)$, $\ldots$, $(p_{l},q_{l})$ in $\sqr{N}$ it holds:
\begin{enumerate}
	\item[(AAU1)] Distinct decay and boundedness condition: For all $\delta_1,\ldots,\delta_l\in\N$ we have
\begin{equation}\label{eq:alphaone}
	\forall\, n\geq N:~\abs{\E a_n(p_1,q_1)^{\delta_1}\cdots a_n(p_l,q_l)^{\delta_l}} \leq\frac{
	C_{\delta_1 + \ldots + \delta_l}
	}{n^{\alpha\cdot\#\{i\in\{1,\ldots,l\}|\delta_i =1\}}},
\end{equation}
\item[(AAU2)]  Second moment condition:
\begin{equation}\label{eq:alphatwo}
\forall\, n\geq N:~\abs{\E a_n(p_1,q_1)^2\cdots a_n(p_l,q_l)^2-1} \leq C^{(l)}_n,
\end{equation}
\end{enumerate}
and \emph{strongly $\alpha$-almost uncorrelated}, if additionally it holds:
\begin{enumerate}
\item[(AAU3)]  Fourth moment condition:
\begin{equation} \label{eq:alphathree}
\forall\, n\geq N:~\abs{\E a_n(p_1,q_1)^4\cdot (a_n(p_2,q_2)^2\cdots a_n(p_l,q_l)^2-1)} \leq D^{(l)}_n,
\end{equation}
\end{enumerate}
Note that further conditions on the speed of convergence of $(C^{(l)}_n)_{n}$ and $(D^{(l)}_n)_{n}$
are made in the statement of the main theorem.
\end{definition}

In order to familiarize oneself with Definition \ref{def:alphaalmostuncorrelated} it is instructive to verify that classical
\emph{Wigner schemes} are $\alpha$-almost uncorrelated for all $\alpha >0$. By a \emph{Wigner scheme}  $(a_n)_n$ we understand a triangular scheme that satisfies the following conditions:
\begin{enumerate}
\item[i)] For all $n\in\N$, the family $(a_n(i,j))_{1\leq i\leq j\leq n}$ is independent.
\item[ii)] For all $n\in\N$ and $(i,j)\in\sqr{n}$: $\E a_n(i,j)=0$ and $\V a_n(i,j)=1$.
\item[iii)] For all $p\in\N$, there is a 
$C'_p>0$
such that $\sup_{n\in\N}\sup_{(i,j)\in\sqr{n}} \E\abs{a_n(i,j)}^p \leq 
C'_p
$.
\end{enumerate}
In Definition \ref{def:alphaalmostuncorrelated} we do not require independence i) of the entries and it is therefore natural to include conditions on products of entries that correspond to fundamentally different index pairs. Condition (AAU1) replaces condition iii) and weakens the requirement of zero expectation ii) to an asymptotic condition where the parameter $\alpha$ governs the decay rate of the correlations. Note that this decay also depends on the number of factors with simple multiplicity. In a similar fashion conditions (AAU2) and (AAU3) correspond to the requirement of unit variance ii). Note that alle the examples we discuss in Section~\ref{sec:examples} are \emph{strongly $\alpha$-almost uncorrelated}. The reason for introducing also a weaker notion is that this weaker form suffices to prove convergence in probability. Finally, we would like to point the reader to Lemma~\ref{lem:gettoaauthree} that provides further insight into condition (AAU3).

Next we recall the notion of a \emph{periodic band matrix}. Let us begin with an example of a $6\times 6$ periodic band matrix $M$ with bandwidth $5$. It has the structure
\[
M=
\begin{pmatrix}
x_{1,1}		&		x_{1,2}		&		x_{1,3}		& 	0				 	&		x_{1,5}		&	  x_{1,6}  	\\
x_{2,1}		&		x_{2,2}		&		x_{2,3}		& 	x_{2,4} 	&			0				&	  x_{2,6}  	\\
x_{3,1}		&		x_{3,2}		&		x_{3,3}		& 	x_{3,4} 	&		x_{3,5}		&	  	0		   	\\
		0			&		x_{4,2}		&		x_{4,3}		& 	x_{4,4} 	&		x_{4,5}		&	  x_{4,6}  	\\
x_{5,1}		&			0				&		x_{5,3}		& 	x_{5,4} 	&		x_{5,5}		&	  x_{5,6}  	\\
x_{6,1}		&		x_{6,2}		&			0				& 	x_{6,4} 	&		x_{6,5}		&	  x_{6,6} 	\\
\end{pmatrix}.
\]
\label{pg:periodicexamples}
Loosely speaking, a band matrix has non-vanishing entries only on some diagonals centered around the main diagonal. The word \emph{periodic} means that $\oneto{n}$ is identified with $\mathbb Z/n\mathbb Z$ so that e.g.~the zeros in the example above are considered to belong to the same diagonal. The bandwidth of the band matrix is the number of non-vanishing diagonals. It can be any odd natural number smaller than $n$. We also include bandwidth equal to $n$ (regardless if $n$ is odd or not), in which case we obtain the full matrix. Therefore, if
$n\in\N$ is arbitrary, then $b_n\in\N$ will be called \emph{($n$-)bandwidth}, if $b_n\in\{b<n\,|\,b \text{ odd}\}\cup\{n\}$.



Given an ($n$-)bandwidth $b_n$ we call index pairs $(i,j)$ \emph{($b_n$-)relevant}, if the corresponding matrix entries $a_n(i,j)$ do not vanish identically. More precisely, for $b_n=n$ all index pairs $(i,j)$ are ($b_n$-)relevant, whereas for $b_n < n$ ($b_n$-)relevance is equivalent to
\[
\abs{i-j}\leq\frac{b_n-1}{2} ~\text{ or } ~\abs{i-j} \geq n-\frac{b_n-1}{2} .
\]
Observe that for $n\in\N$ and ($n$-)bandwidth $b_n$ there are exactly $b_n$ relevant index pairs in each row, yielding a total of  $n \cdot b_n$ relevant index pairs in $\sqr{n}$.

Any $n\times n$ matrix can be converted to a periodic band matrix with given bandwidth $b_n$ by setting the non-relevant entries equal to $0$: For triangular schemes $(a_n)_n$ and sequences of $n$-bandwidths $b=(b_n)_n$ we define the periodic band matrices $a_n^{b}$ by
\[
\forall\, n\in\N:\,\forall\,(i,j)\in\sqr{n}:~ a_n^{b}(i,j)\defeq
\begin{cases}
a_n(i,j) & \mbox{if $(i,j)$ is $b_n$-relevant}\\
0 & \mbox{otherwise.}
\end{cases}
\]

\begin{definition}
\label{def:randommatrix}
Let $(\Omega,\Acal,\Prob)$ be a probability space, $(a_n)_{n\in\N}$ a triangular scheme and $b=(b_n)_n$ a sequence of $n$-bandwidths.
We say that a sequence of periodic random band matrices $(X_n)_{n\in\N}$ is \emph{based on the triangular scheme $(a_n)_{n\in\N}$ with bandwidth $b$}, if for all $n\in\N$ we have:
	\begin{enumerate}[i)]
		\item $X_n$ has dimension $n\times n$.
		\item $X_n(p,q)=\frac{1}{\sqrt{b_n}} a^{b}_n(p,q) \quad\forall\, (p,q)\in\sqr{n}$.
	\end{enumerate}
\end{definition}

\begin{remark}
Please note that the important and often treated case of full matrices is included in the above definition by setting $b_n=n$.
\end{remark}


Recall from the beginning of the Introduction the definition of the empirical spectral distribution (ESD) $\sigma_n$ which defines a
random probability measure on $(\R,\Bcal)$, the real line equipped with the Borel sigma algebra. The measures $\sigma_n$ are random (i.e., stochastic kernels) as they depend on the realizations of the random matrix entries. For the ensembles considered we prove convergence of these random measures to a deterministic measure called the \emph{semicircle distribution $\sigma$}, defined by its Lebesgue density $f_\sigma$ with

\[
\forall\,x\in\R:~f_\sigma(x) \defeq \frac{1}{2\pi}\sqrt{4-x^2}\one_{[-2,2]}(x).
\]
We define different types of convergence that are relevant for our results.


\begin{definition}\label{def:randomweakconvergence}
Let $\sigma$ be the semicircle distribution on $(\R,\Bcal)$ and denote by $\boundedcont{\R}$ the set of bounded continuous functions $\R\to\R$. Then the empirical spectral distributions $\sigma_n$ of random matrices $X_n$ are said to converge
\begin{enumerate}[i)]
	\item \emph{weakly in expectation} to $\sigma$, if
	\(
	\forall\, f\in\boundedcont{\R}: \E\integrala{\sigma_n}{f}\xrightarrow[n\to\infty]{} \integrala{\sigma}{f},
	\)
	\item \emph{weakly in probability} to $\sigma$, if
	\(
	\forall\, f\in\boundedcont{\R}: \integrala{\sigma_n}{f}  \xrightarrow[n\to\infty]{} \integrala{\sigma}{f} \text{ in probability},
	\)
	\item \emph{weakly almost surely} to $\sigma$, if
		\(
	\forall\, f\in\boundedcont{\R}: \left[\integrala{\sigma_n}{f}\xrightarrow[n\to\infty]{} \integrala{\sigma}{f} \text{ almost surely}\right],	
	\)
\end{enumerate}
where for any measure $\nu$ on $(\R,\Bcal)$ and a $\nu$-integrable function $g:(\R,\Bcal)\to(\R,\Bcal)$, we define
\[
\integrala{\nu}{g}\defeq \integrald{x}{\nu}{g(x)}{\R}.
\]
\end{definition}

The following remarks that shed some more light on these three types of convergence can be found e.g.~in \autocite{FleermannDiss}.
\begin{enumerate}[i)]

\item is the deterministic weak convergence of $\E\sigma_n$ to $\mu$, where $\E\sigma_n$ denotes the probability measure on $(\R,\Bcal)$ given by $(\E\sigma_n)(B)\defeq \E[\sigma_n(B)]$ for all $B\in\Bcal$.

\item If $d$ is \emph{any} metric on the space of probability measures on $(\R,\Bcal)$ that metrizes weak convergence (many such metrics exist), then $\sigma_n\to\mu$ weakly
in probability holds iff $d(\sigma_n,\mu)\to 0$ in probability.

\item If $d$ metrizes weak convergence, then $\sigma_n\to\mu$
weakly almost surely holds iff almost surely
$d(\sigma_n,\mu)\to 0$. This in turn is equivalent to the statement that almost surely we have $[\forall\, f\in\boundedcont{\R}: \integrala{\sigma_n}{f}\to \integrala{\mu}{f}]$. To appreciate the last equivalence, note that $\boundedcont{\R}$ is not separable.
\end{enumerate}

Now if $(\sigma_n)_n$ are the ESD of the random matrices $(X_n)_n$ and we have that $\sigma_n\to\sigma$ weakly in expectation/ in probability/ almost surely, then we say that \emph{the semicircle law holds for $(X_n)_n$} in expectation/ in probability/ almost surely.

We are now ready to state our main result that is proved in Section~\ref{sec:mainproof}.

\begin{theorem}\label{thm:main}
Let $(a_n)_n$ be an $\alpha$-almost-uncorrelated triangular scheme, $b=(b_n)_n$ be a sequence of $n$-bandwidths with $b_n\to\infty$ and $(X_n)_n$ be the periodic random band matrices which are based on $(a_n)_n$ with bandwidth $b$. Then we obtain the following results:
\begin{enumerate}
\item If $\alpha \geq \frac{1}{2}$, then the semicircle law holds for $(X_n)_n$ in probability.
\item If $\alpha\geq \frac{1}{2}$, $\frac{1}{b_n^3}$ is summable over $n$ and all entries of $(a_n)_n$ are $\{-1,1\}$-valued, then the semicircle law holds almost surely for $(X_n)_n$.
\item If $(a_n)_n$ is even strongly $\alpha$-almost-uncorrelated with $\alpha>\frac{1}{2}$, and the sequences $\frac{1}{b_n^{2}}$, $\frac{1}{b_n}D^{(l)}_n$ and $C^{(l)}_n$ are summable over $n$ for all $l\in\N$, then we obtain the semicircle law almost surely for $(X_n)_n$.
\item If $(a_n)_n$ is a Wigner scheme and if $(\frac{1}{nb_n})_n$ is summable, then we obtain the semicircle law almost surely for $(X_n)_n$.
\end{enumerate}
\end{theorem}

\begin{remark}
Observe that all statements of Theorem~\ref{thm:main} that allow to deduce the almost sure semicircle law require minimal growth conditions on the bandwidths.

In statement 3 the condition on $D^{(l)}_n$ depends on the growth behavior of the bandwidths. However, square summability of $D^{(l)}_n$ always suffices by the Cauchy-Schwarz inequality.

If the bandwidths grow at least at a linear rate, e.g.~for full matrices, the conditions that imply almost sure convergence simplify.  This is the content of the subsequent corollary.
\end{remark}




\begin{corollary}\label{cor:propband}
Let $(a_n)_n$ be an $\alpha$-almost-uncorrelated triangular scheme, $b=(b_n)_n$ a bandwidth
satisfying $\liminf_{n \to \infty} b_n/n > 0$.
Let $(X_n)_n$ be the periodic random band matrices which are based on $(a_n)_n$ with bandwidth $b$.
(This includes in particular the case of full matrices with no diagonals being set to equal $0$.) Then the semicircle law holds almost surely for $(X_n)_n$ if one of the following additional three conditions is satisfied:
\begin{itemize}
\item[2\,$'$.] $\alpha\geq \frac{1}{2}$ and all entries of $(a_n)_n$ are $\{-1,1\}$-valued.
\item[3\,$'$.] $(a_n)_n$ is strongly $\alpha$-almost-uncorrelated with $\alpha>\frac{1}{2}$ and the sequences $C^{(l)}_n$ and $\frac{1}{n}D^{(l)}_n$ are both summable over $n$ for all $l\in\N$.
\item[4\,$'$.] $(a_n)_n$ is a Wigner scheme.
\end{itemize}
\end{corollary}
\begin{proof}
The hypothesis implies the existence of a positive constant $c$ with $b_n/n > c$ for all $n$ large. The corollary then follows from the corresponding statements of Theorem~\ref{thm:main}.
\end{proof}
Note that the last corollary contains a version of the classical Wigner semicircle law.


\section{Examples}\label{sec:examples}

In this section we treat two classes of $\alpha$-almost uncorrelated triangular schemes and apply Theorem~\ref{thm:main} to prove the semicircle law under suitable conditions.

\subsection{Curie-Weiss Ensembles}

Almost uncorrelated ensembles were first introduced in \autocite{HKW} for the analysis of random matrices with Curie-Weiss distributed entries. Let us start by recalling their definition of approximately uncorrelated triangular schemes \autocite[Definition~4]{HKW}:

\begin{definition}
\label{def:almostuncorrelated}
A triangular scheme $(a_n)_{n\in\N}$ is called \emph{approximately uncorrelated}, if for all $N \in\N$ we have:
\begin{enumerate}
\item[(AU1)]\label{item:AU1} Distinct decay and boundedness condition:
\[
 \forall\, n\geq N:~\abs{\E a_n(p_1,q_1)\cdots a_n(p_s,q_s) a_n(p_{s+1},q_{s+1})\cdots a_n(p_{s+m},q_{s+m})}\leq\frac{K_{s,m}}{n^{s/2}}\label{eq:auone} \\
\]
\item[(AU2)]\label{item:AU2} Second moment condition:
\[
\forall\, n\geq N:~\abs{\E a_n(p_1,q_1)^2\cdots a_n(p_s,q_s)^2-1} \leq K^{(s)}_n\label{eq:autwo}
\]
\end{enumerate}
for all sequences of pairs $(p_1,q_1),\ldots,(p_{s+m},q_{s+m})$ in $\sqr{N}$, where $s,m\in\N_0$, so that $(p_1,q_1),\ldots$, $(p_s,q_s)$ are pairwise fundamentally different and fundamentally different from all other pairs of the sequence $(p_{s+1},q_{s+1}),\ldots,(p_{s+m},q_{s+m})$.
Further, the constants $K_{s,m}$ are non-negative real numbers that only depend on $s$ and $m$, and for all $s\in\N_0$ we have that $(K^{(s)}_n)_{n\in\N}$ is a non-negative real sequence that converges to zero.
\end{definition}

One may identify the roles that $s$ and $s+m$ play in condition (AU1) of Definition~\ref{def:almostuncorrelated} with the terms $\#\{i\in\{1,\ldots,l\}|\delta_i =1\}$ and $\delta_1+\ldots +\delta_l$ appearing in condition (AAU1) of Definition~\ref{def:alphaalmostuncorrelated} respectively. Then the following statement of equivalence is not difficult to prove.


\begin{lemma}\label{lem:translation}
A triangular scheme $(a_n)_n$ is approximately uncorrelated iff it is $\frac{1}{2}$-almost uncorrelated.
\end{lemma}



 \begin{remark}\label{rem:comparisonHKW}
 Observe that Lemma~\ref{lem:translation} together with claim 1 of Theorem~\ref{thm:main} proves the semicircle law in probability for periodic random band matrices based on an approximately uncorrelated triangular scheme $(a_n)_n$ with bandwidths $b_n\to\infty$. This generalizes Theorem~5 of \autocite{HKW} to band matrices. As we explain after the present remark, Theorem~31 of \autocite{HKW} that deals with matrices with Curie-Weiss distributed entries (and a certain generalization thereof) is extended to periodic band matrices by the second claim of Theorem~\ref{thm:main}. Moreover, it sharpens \autocite[Theorem~31]{HKW} by lifting convergence in probability to almost sure convergence without changing the hypotheses.
 \end{remark}

 We now recall the definition of \emph{Curie-Weiss ensembles} that was introduced in \autocite{HKW}. Note that in \autocite{HKW} these are called \emph{full Curie-Weiss ensembles} in order to distinguish them from the ensembles studied in \autocite{Friesen:zwei}.


\begin{definition}\label{def:curieweiss}
Let $n\in\N$ be arbitrary and $Y_1,\ldots,Y_n$ be random variables defined on some probability space $(\Omega,\Acal,\Prob)$. Let $\beta>0$, then we say that $Y_1,\ldots,Y_n$ are Curie-Weiss($\beta$,$n$)-distributed\label{sym:CurieWeissdist}, if for all $y_1,\ldots,y_n\in\{-1,1\}$ we have that
\[
\Prob(Y_1=y_1,\ldots,Y_n=y_n) = \frac{1}{Z_{\beta,n}}\cdot e^{\frac{\beta}{2n}\left(\sum y_i\right)^2}
\]
where $Z_{\beta,n}$\label{sym:CWconstant} is a normalization constant. The parameter $\beta$ is called \emph{inverse temperature}.
\end{definition}

The Curie-Weiss($\beta,n$) distribution is used to model the behavior of $n$ ferromagnetic particles (with spins $y_j$) at inverse temperature $\beta$. At low temperatures, that is large values of $\beta$, all spins are likely to have the same alignment, modelling strong magnetization. At high temperatures, however, the spins become almost independent, resembling weak magnetization. We refer the reader to \autocite{KirschMomentSurvey} for a self-contained treatment of the Curie-Weiss distribution.

\begin{theorem}\label{thm:CurieWeiss}
Let $0<\beta\leq 1$ and let the random variables $(\tilde{a}_n(i,j))_{1\leq i,j \leq n}$ be Curie-Weiss($\beta,n^2$)-distributed for each $n\in\N$ . Define the triangular scheme $(a_n)_n$ by setting
\[
\forall\,n\in\N:\,\forall\, (i,j)\in\sqr{n}:~a_n(i,j) =
\begin{cases}
\tilde{a}_n(i,j) & \text{if $i\leq j$}\\
\tilde{a}_n(j,i) & \text{if $i> j$}.
\end{cases}
\]
 Let $b=(b_n)_n$ be a sequence of $n$-bandwidths and $(X_n)_n$ be the periodic random band matrices which are based on $(a_n)_n$ with bandwidth $b$. Then the following statements hold:
 \begin{enumerate}[i)]
 \item The triangular scheme $(a_n)_n$ is $\frac{1}{2}$-almost uncorrelated.
 \item If $b_n\to\infty$ as $n\to\infty$, then the semicircle law holds for $(X_n)_n$ in probability.
 \item If $\frac{1}{b_n^3}$ is summable over $n$, then the semicircle law holds almost surely for $(X_n)_n$. Observe that this statement implies in particular to full matrices ($b_n=n$).
 \end{enumerate}

\end{theorem}
\begin{proof}
In \autocite{HKW} it was shown that 	$(a_n)_n$ is approximately uncorrelated.
Thus i) follows from Lemma~\ref{lem:translation}. Statements ii) and iii) follow from the first two claims of Theorem~\ref{thm:main}.
\end{proof}

\subsection{Correlated Gaussian Entries}

We now study random matrices filled with correlated Gaussian entries. By placing quite natural conditions on their covariance matrices
we obtain $\alpha$-almost uncorrelated ensembles.
We begin with a lemma that helps to validate the forth moment condition (AAU3) that is essential for proving the almost sure semicircle law.

\begin{lemma}\label{lem:gettoaauthree}
Let $(a_n)_{n\in\N}$ be a triangular scheme and suppose that there exists a $K\in\R$ such that for all $l,N\in\N$ and fundamentally different pairs $(p_1,q_1),\ldots,(p_{l},q_{l})$ in $\sqr{N}$ we have that
\begin{equation}\label{eq:gettoaauthree}
\forall n\geq N:~\abs{\E a_n(p_1,q_1)^4 a_n(p_2,q_2)^2\cdots a_n(p_l,q_l)^2 - K} \leq \tilde{D}^{(l)}_n,
\end{equation}
where for each $l\in\N$, $(\tilde{D}^{(l)}_n)_n$ is a sequence converging to zero. Then (AAU3) is also satisfied with constants $D^{(l)}_n\defeq\tilde{D}^{(l)}_n+\tilde{D}^{(1)}_n$.
\end{lemma}
\begin{proof}
The lemma follows from the basic estimate
\begin{multline*}
\abs{\E a_n(p_1,q_1)^4\cdot (a_n(p_2,q_2)^2\cdots a_n(p_l,q_l)^2-1)}\\ \leq \abs{\E a_n(p_1,q_1)^4 a_n(p_2,q_2)^2\cdots a_n(p_l,q_l)^2-K} + \abs{\E a_n(p_1,q_1)^4 -K}
\leq \tilde{D}^{(l)}_n + \tilde{D}^{(1)}_n.
\end{multline*}
\end{proof}

What follows is a generalization of the first author's work in \autocite{FleermannMT}. For any $\alpha>0$ we denote by $\CovMat(\alpha)$\label{sym:specialcovmat} the set of all sequences $(\Sigma_n)_n$, where for each $n\in\N$, $\Sigma_n$ is a real symmetric $n\times n$ matrix with the following properties:
\begin{enumerate}[i)]
	\item $\Sigma_n(i,i) = 1$ for all $i\in\oneto{n}$,
	\item $\abs{\Sigma_n(i,j)} \leq 1/n^{\alpha}$ for all $i\neq j\in\oneto{n}$.
	\item $\Sigma_n$ is positive definite.
\end{enumerate}

Note for $\alpha\geq 1$ that any symmetric $n\times n$ matrix $A$ satisfying above conditions i) and ii) is \emph{strictly diagonally dominant} with strictly positive diagonal entries and therefore satisfies also condition iii) (see e.g.~\autocite{Plato}). The next theorem is known as "Wick's theorem" or "Theorem of Isserlis" (see e.g.~\autocite{Speicher} or \autocite{Isserlis}). The theorem is useful for computing the expectation of an arbitrary product of multi-dimensional normal random variables.




\begin{theorem}
\label{thm:wick}
Let $n\in\N$ and $\Sigma$ be a positive definite, real symmetric $n\times n$ matrix. For the real-valued random variables $Y_1,\ldots, Y_n$ on $(\Omega,\Acal,\Prob)$ it is assumed that $(Y_1,\ldots,Y_n)\sim \mathcal{N}(0_n,\Sigma)$ (where $0_n$ is the n-dimensional vector containing only zeroes). Then for all $k\in\N$ and $i(1),\ldots,i(k)\in\oneto{n}$, we have that
\[
\E Y_{i(1)}\cdots Y_{i(k)} = \sum_{\pi\in\PePe{k}}\prod_{\{r,s\}\in\pi}\E Y_{i(r)}Y_{i(s)} = \sum_{\pi\in\PePe{k}}\prod_{\{r,s\}\in\pi} \Sigma \big(i(r),i(s)\big),
\]
where $\PePe{k}$\label{sym:pairpartitions} denotes the set of all pair partitions on $\{1,\ldots,k\}$. Especially, we obtain for $k$ odd that
\[
\E Y_{i(1)}\cdots Y_{i(k)} = 0.
\]
\end{theorem}

Next we define a class of triangular schemes $(a_n)_{n\in\N}$ built from correlated Gaussian entries where all correlations between different entries are bounded in modulus by $2^{\alpha}/n^{2\alpha}$ for some prescribed parameter $\alpha>0$. Note that, unlike the class of Curie-Weiss ensembles, this class features entries which are not necessarily exchangeable random variables. This is also the reason why we are more careful in defining how the random variables fill the triangular scheme.

\begin{example}
\label{ex:fubsp}
Let $\alpha>0$ be arbitrary.
Due to symmetry, it suffices to specify the right upper triangle of each $a_n$ in the triangular scheme $(a_n)_{n\in\N}$ that has $\Delta_n \defeq n(n+1)/2$ entries.
Fix a sequence $(\Sigma_n)_n\in\CovMat(\alpha)$.
We fill the right upper triangle of each $a_n$ using all the entries from a random vector $(Y^{(n)}_1,\ldots,Y^{(n)}_{\Delta_n})\sim \mathcal{N}(0_{\Delta_n},\Sigma_{\Delta_n})$.
To this end, for each $n\in\N$ we fix a map
$\varphi_n : \sqr{n}  \longrightarrow \oneto{\Delta_n}$ such that $\varphi_n$ restricted to $\{(i,j)\in\sqr{n} : i\leq j\}$ is a bijection and such that $\varphi_n(i,j) = \varphi_n(j,i)$ for all $(i,j)\in\sqr{n}$ with $j<i$. The triangular scheme is then defined via $a_n(i,j)\defeq Y^{(n)}_{\varphi_n(i,j)}$ for all $(i,j)\in\sqr{n}$.

\end{example}

The next three lemmas help us to show that the just defined triangular schemes are strongly $\alpha$-almost uncorrelated.


\begin{lemma}
\label{lem:bspfueins}
Let  $l\in \N$, $\alpha>0$, and $(\Sigma_n)_n\in\CovMat(\alpha)$ be arbitrary. Then choose $\delta_1,\ldots,\delta_l\in\N$ so that $\delta_1+\ldots+\delta_l$ is even. Furthermore, we fix $n\in\N$ with $l \leq \Delta_n$ and pick integers
$i(1),i(2),\ldots,i(\delta_1+\ldots+\delta_l) \in \oneto{\Delta_n}$ so that
\begin{align*}
&i(1)=i(2)=\ldots=i(\delta_1)\\
&i(\delta_1+1)=\ldots= i(\delta_1+\delta_2)\\
&i(\delta_1+\delta_2+1)=\ldots=i(\delta_1+\delta_2+\delta_3)\\
&\vdots\\
&i(\delta_1+\ldots+\delta_{l-1}+1)=\ldots=i(\delta_1+\ldots+\delta_l)
\end{align*}
but $i(1)$, $i(\delta_1+1)$, $\ldots$, $i(\delta_1 + \ldots + \delta_{l-1}+1)$ are pairwise distinct. Then we have
\[
\bigabs{\sum_{\pi\in\PePe{\delta_1+\ldots+\delta_l}}\prod_{\{r,s\}\in\pi} \Sigma_{\Delta_n}\big(i(r),i(s)\big)} \leq  \frac{\#\PePe{\delta_1+\ldots+\delta_l}}{(n/\sqrt{2})^{\alpha\cdot\#\{i\,|\,\delta_i =1\}}}.
\]
\end{lemma}
\begin{proof}
Each diagonal entry of $\Sigma_{\Delta_n}$ equals $1$ and each off-diagonal entry lies in the interval $[-1/\Delta_n^{\alpha}, 1/\Delta_n^{\alpha}]$. Let $\pi\in\PePe{\delta_1+\ldots+\delta_l}$ be arbitrary, then we have that
$\prod_{\{r,s\}\in\pi} \Sigma_{\Delta_n}(i(r),i(s))$
is a product of  $(\delta_1+\ldots+\delta_l)/2$ entries of $\Sigma_{\Delta_n}$. For each block $\{r,s\}\in\pi$ with $i(r)=i(s)$ we find $\Sigma_{\Delta_n}\big(i(r),i(s)\big)=1$, and for each block $\{r,s\}\in\pi$ with $i(r)\neq i(s)$ we obtain $\abs{\Sigma_{\Delta_n}\big(i(r),i(s)\big)}\leq 1/\Delta_n^{\alpha}$. But now we have at least
$\#\{i\,|\,\delta_i = 1\}/2$
blocks $\{r,s\}$ in $\pi$ with $i(r)\neq i(s)$, since for any $\delta_i$ with $\delta_i=1$, the index $i(\delta_1+\ldots+\delta_i)$ is unique among all indices and must therefore share a block with a different index. Then in the worst case possible, the number of $\delta_i$'s with $\delta_i=1$ is even and their corresponding unique indices are all paired, yielding the lower bound on the number of non-diagonal entries in the product. As $\Delta_n > n^2/2$ we arrive at
\[
\prod_{\{r,s\}\in\pi} \abs{\Sigma_{\Delta_n}\big(i(r),i(s)\big)} \leq \left(\frac{1}{\Delta_n^{\alpha}}\right)^{\frac{\#\{i|\delta_i =1\}}{2}} \leq\; \frac{1}{(n/\sqrt{2})^{\alpha\cdot\#\{i\,|\,\delta_i =1\}}} \,.
\]
This bound holds for each $\pi\in\PePe{\delta_1+\ldots+\delta_l}$, proving the lemma.
\end{proof}

\begin{lemma}
\label{lem:bspfuzwei}
Let $\alpha>0$ and $(\Sigma_n)_n\in\CovMat(\alpha)$ be arbitrary. Let $n, z\in \N$ satisfy $z \leq \Delta_n$. Choose $i(1),\ldots,i(2z)$ in $\oneto{\Delta_n}$, so that $i(1)=i(2), i(3)=i(4), \ldots, i(2z-1)=i(2z)$, and so that $i(1),i(3),\ldots,i(2z-1)$ are pairwise distinct. Then it holds:
\[
\bigabs{\sum_{\pi\in\PePe{2z}}\prod_{\{r,s\}\in\pi} \Sigma_{\Delta_n}\big(i(r),i(s)\big) - 1} \leq \frac{\#\PePe{2z}}{(n/\sqrt{2})^{4\alpha}}.
\]
\end{lemma}
\begin{proof}
Consider the pair partition $\pi_0 = \left\{\{1,2\},\{3,4\},\ldots,\{2z-1,2z\}\right\}\in\PePe{2z}$. This is the only pair partition in $\PePe{2z}$
so that only diagonal entries of $ \Sigma_{\Delta_n}$ appear in the product.
Thus $\prod_{\{r,s\}\in\pi_0} \Sigma_{\Delta_n}\big(i(r),i(s)\big) = 1$.
For each $\pi\in\PePe{2z}$ with $\pi\neq\pi_0$ we find a block $\{r',s'\}\in\pi$ with $i(r')\neq i(s')$, and then necessarily at least one further block $\{r'',s''\}\in\pi$ with $i(r'')\neq i(s'')$, leading to
$\prod_{\{r,s\}\in\pi} \abs{\Sigma_{\Delta_n}\big(i(r),i(s)\big)}\leq \frac{1}{\Delta_n^{2\alpha}}$.There are at most $\#\PePe{2z}$ partitions $\pi\in\PePe{2z}$ with $\pi\neq\pi_0$, which concludes the proof.
\end{proof}

\begin{lemma}
\label{lem:bspfudrei}
Let $\alpha>0$ and $(\Sigma_n)_n\in\CovMat(\alpha)$ be arbitrary.
 Let $n, z\in \N$ with $z \leq \Delta_n$. Choose $i(1),\ldots,i(2z+2)$ in $\oneto{\Delta_n}$, so that $i(1)=i(2)=i(3)=i(4), i(5)=i(6), i(7) = i(8), \ldots, i(2z+1)=i(2z+2)$, and so that $i(1),i(5),i(7),\ldots,i(2z+1)$ are pairwise distinct. Then it holds:
\[
\bigabs{ \sum_{\pi\in\PePe{2z+2}}\prod_{\{r,s\}\in\pi} \Sigma_{\Delta_n}\big(i(r),i(s)\big) - 3} \leq \frac{\#\PePe{2z+2}}{(n/\sqrt{2})^{4\alpha}}.
\]
\end{lemma}
\begin{proof}
The proof is analogous to the proof of Lemma~\ref{lem:bspfuzwei}. Here, the pair partitions
\begin{align*}
&\pi_1 = \left\{\{1,2\},\{3,4\},\{5,6\},\{7,8\},\ldots,\{2z+1,2z+2\}\right\}\\
&\pi_2 = \left\{\{1,3\},\{2,4\},\{5,6\},\{7,8\},\ldots,\{2z+1,2z+2\}\right\}\\
&\pi_3 = \left\{\{1,4\},\{2,3\},\{5,6\},\{7,8\},\ldots,\{2z+1,2z+2\}\right\}\\
\end{align*}
yield a summand of $1$, whereas all other partitions yield a summand of absolute value $\leq\Delta_n^{-2\alpha}$.
\end{proof}

Now we are ready to prove that Example~\ref{ex:fubsp} indeed provides strongly $\alpha$-almost uncorrelated triangular schemes.

\begin{theorem}\label{thm:fubsp}
Let $\alpha>0$ and $(\Sigma_n)_{n}\in\CovMat(\alpha)$ be arbitrary. Let the triangular scheme $(a_n)_{n\in\N}$ be constructed with respect to $\alpha$ and $(\Sigma_n)_{n}$ as described in Example~\ref{ex:fubsp}. Then  $(a_n)_{n\in\N}$ is strongly $\alpha$-almost uncorrelated, and if $\alpha>1/4$, then for all $l\in\N$, the sequences $C^{(l)}_n$ and $D^{(l)}_n$ can be chosen summable over $n$. Further, this property is tight in the sense that $(a_n)_{n\in\N}$ need not be $\alpha'$-almost uncorrelated for any $\alpha'>\alpha$.
\end{theorem}
\begin{proof}
To prove property (AAU1) we only need to consider the case where $\delta_1+\ldots +\delta_l$ is even because of the last statement in
 Theorem~\ref{thm:wick}. So let $N\in\N$ and $l\in\N$ be arbitrary, choose $\delta_1,\ldots,\delta_l\in\N$ so that $\delta_1+\ldots +\delta_l$ is even, and fix a sequence of fundamentally different pairs $(p_1,q_1),\ldots,(p_l,q_l)\in\sqr{N}$. Then the indices $i(1),i(2),\ldots,i(\delta_1+\ldots+\delta_n)$ defined through
\begin{align*}
&i(1), i(2), \ldots, i(\delta_1) \defeq \varphi_n(p_1,q_1)\\
&i(\delta_1+1), i(\delta_1 +2), \ldots, i(\delta_1 +\delta_2) \defeq \varphi_n(p_2,q_2)\\
&\vdots\\
&i(\delta_1+\ldots + \delta_{l-1} + 1), \ldots, i(\delta_1+\ldots + \delta_{l-1} + \delta_{l}) \defeq \varphi_n(p_l,q_l)
\end{align*}
meet the conditions of Lemma~\ref{lem:bspfueins}. With Lemma~\ref{lem:bspfueins}, Theorem~\ref{thm:wick} and the definition $a_n(p_i,q_i)=Y^{(n)}_{\varphi_n(p_i,q_i)}$ of Example~\ref{ex:fubsp} we conclude
\begin{eqnarray*}
\bigabs{\E a_n(p_1,q_1)^{\delta_1}\cdot a_n(p_{2},q_{2})^{\delta_2}\cdot \ldots \cdot a_n(p_{l},q_{l})^{\delta_l}}
&=&\bigabs{\E  Y^{(n)}_{i(1)}\cdots Y^{(n)}_{i(\delta_1+\ldots+\delta_{l})}}\\
= \bigabs{\sum_{\pi\in\PePe{\delta_1+\ldots+\delta_l}}\prod_{(r,s)\in\pi} \Sigma_{\Delta_n}(i(r),i(s))}&\leq&\frac{\#\PePe{\delta_1+\ldots+\delta_l}}{(n/\sqrt{2})^{\alpha\cdot\#\{j\,|\,\delta_j=1\}}}
\end{eqnarray*}
Thus (AAU1) holds with
 $C_{k} \defeq  \#\PePe{k} 2^{\alpha k/2}$.

Condition (AAU2) is shown similarly by using Lemma~\ref{lem:bspfuzwei}, Theorem~\ref{thm:wick} and by defining $C^{(l)}_n\defeq 4^{\alpha}\frac{\#\PePe{2l}}{n^{4\alpha}}$, which are summable over $n$ iff $\alpha>1/4$.
For condition (AAU3) we use that Lemma~\ref{lem:bspfudrei} implies by a similar reasoning that the hypotheses of Lemma~\ref{lem:gettoaauthree} are satisfied with $K = 3$ and  $\tilde{D}^{(l)}_{n} = 4^{\alpha} \frac{\# \PePe{2l+2}}{n^{4\alpha}}$. Hence (AAU3) holds with $D^{(l)}_n=4^{\alpha} \frac{\# \PePe{2l+2} + \# \PePe{4}}{n^{4\alpha}}$, which again is summable over $n$ iff $\alpha>1/4$.

To prove the last claim of the theorem choose  $\Sigma_n(i,j)=1/n^{\alpha}$ for all $n\in\N$ and $i\neq j\in\oneto{n}$. Observe that $\Sigma_n$ is positive definite as a sum of a positive semi-definite matrix and a positive multiple of the identity matrix. We obtain $\abs{\E a_n(1,1)a_n(1,2)}  = \Delta_n^{-\alpha}$ for all $n\geq 2$. Thus (AAU1) cannot hold for any $\alpha'>\alpha$ in the case $l=2$ with $\delta_1=\delta_2=1$.
\end{proof}


The following theorem summarizes our main results for correlated Gaussian entries. These are novel results when compared with previous work on Gaussian entries, such as \autocite{MonvelKhorunzhy1999}, \autocite{BannaMerlevede2015} and \autocite{ChakrabartyHazra2016}, where different correlation structures are assumed.

\begin{theorem}
\label{cor:fubsp}
Let $\alpha>0$ and $(\Sigma_n)_n\in\CovMat(\alpha)$ be arbitrary, and let $(a_n)_n$ be the triangular scheme filled with correlated Gaussian entries as described in Example~\ref{ex:fubsp}.
Let $b=(b_n)_n$ be a sequence of bandwidths with $b_n\to\infty$ and $(X_n)_n$ be the periodic random band matrices based on $(a_n)_n$ with bandwidth $b$. Then the following statements hold:

\begin{enumerate}[i)]
\item If $\alpha\geq 1/2$, then the semicircle law holds in probability for $(X_n)_n$.
\item If $\alpha>1/2$, and $(\frac{1}{b_n^2})_n$ is summable then the semicircle law holds almost surely for $(X_n)_n$. Observe that this statement implies in particular to full matrices ($b_n=n$).
\end{enumerate}
\end{theorem}
\begin{proof}
By Theorem~\ref{thm:fubsp}, $(a_n)_{n\in\N}$ is strongly $\alpha$-almost uncorrelated. For $\alpha>1/4$ the sequences $C^{(l)}_n$ and $D^{(l)}_n$ can be chosen to be summable over $n$ for all $l\in\N$. Then both claims follow directly from statements 1.\ and 3.\ of Theorem~\ref{thm:main}.
\end{proof}

\section{Proof of Theorem~\ref{thm:main}}\label{sec:mainproof}

Our main result of this paper, Theorem~\ref{thm:main} follows directly from Theorems~\ref{thm:expectationandvariance},~\ref{thm:expectationtocatalan}, and~\ref{thm:variancetozero}  below. The proof is based on the method of moments, a principle which was already used by E.\ Wigner \autocite{Wigner} in his proof of the semicircle law. Since then refined versions of Wigner`s argument have continually led to new results in random matrix theory, even at the sophisticated level of local laws \autocite{Soshnikov1999}. See e.g.\ \autocite{SixtyYears, Sodin2014} for recent overviews about the applications of the moment method to random matrix theory.

The method of moments is based on the fact that if $\mu$ is uniquely determined by its moments and $(\mu_n)_n$ is a sequence of probability measures with $\integrala{\mu_n}{x^k}\to \integrala{\mu}{x^k}$ for all $k\in\N$,
then $\mu_n\to\mu$ weakly. This notion carries over one-to-one to all probabilistic weak convergence types as in Definition~\ref{def:randomweakconvergence} (see e.g.\ \autocite{FleermannDiss} for a proof): If $(\sigma_n)_n$ is a sequence of ESDs, then if for all $k\in\N$, $\integrala{\sigma_n}{x^k}\to_n\integrala{\sigma}{x^k}$ in expectation resp.\ in probability resp.\ almost surely, then also $\sigma_n\to\sigma$ weakly in expectation resp.\ in probability resp.\ almost surely. With these insights, the following theorem -- which we will use for our proof -- is immediate:


\begin{theorem}\label{thm:expectationandvariance}
Let $(\sigma_n)_n$ be the empirical spectral distributions of random matrices $(X_n)_n$, whose entries have moments of all orders. Denote by $\sigma$ the semicircle distribution. Then 
\begin{enumerate}[i)]
	\item $\sigma_n$ converges to $\sigma$ weakly in expectation, if
	\begin{equation}\label{eq:toshowfirst}
	\forall\,k\in\N: \E\integrala{\sigma_n}{x^k} \xrightarrow[n\to\infty]{} \integrala{\sigma}{x^k},
	\end{equation}
	\item $\sigma_n$ converges to $\sigma$ weakly in probability, if i) holds and
	\begin{equation}\label{eq:toshowsecond}
	\forall\,k\in\N: \V\integrala{\sigma_n}{x^k} \xrightarrow[n\to\infty]{} 0,
	\end{equation}
	\item $\sigma_n$ converges to $\sigma$ weakly almost surely, if i) holds and
	\[
	\forall\,k\in\N: \V\integrala{\sigma_n}{x^k} \xrightarrow[n\to\infty]{} 0 \quad \text{summably fast.}
	\]
\end{enumerate}
\end{theorem}

Since all random variables in $\alpha$-almost uncorrelated triangular schemes have moments of all orders, Theorem~\ref{thm:expectationandvariance} is applicable to the ensembles in this paper.
The usefulness of Theorem~\ref{thm:expectationandvariance}
is based on two observations.
First,
the moments of the semicircle distribution can be expressed in terms of the Catalan numbers $(\Cat_n)_{n\in\N_0}$ that are given by $\Cat_n = \frac{1}{n+1} \binom{2n}{n}$. Indeed,
\begin{equation}\label{eq:momentsofsemicircle}
\forall\,k\in\N: \integrala{\sigma}{x^k}=
\begin{cases}
\Cat_{\frac{k}{2}} 	& \text{ for even $k$,}\\
0 									&	\text{ for odd $k$.}
\end{cases}
\end{equation}
Second, for the ensembles $(X_n)_n$ considered in Theorem~\ref{thm:main}, the moments of the empirical spectral distributions $\sigma_n$ satisfy

\begin{equation}\label{eq:momenttosum}
\integrala{\sigma_n}{x^k} = \frac{1}{n}\tr (X^k_n)=\frac{1}{nb_n^{k/2}}\sum_{t_1,\ldots,t_k=1}^{n}{a^b_n(t_1,t_2)a^b_n(t_2,t_3)\cdots a^b_n(t_k,t_1)}.
\end{equation}
The remainder of this section is subdivided in three parts.
First we develop the combinatorics to deal with the sum in \eqref{eq:momenttosum}. Then we show
\eqref{eq:toshowfirst}. Lastly, we prove \eqref{eq:toshowsecond} and investigate in which cases the convergence is summably fast.

\subsection{Combinatorics}
In this subsection we present a refined version of the combinatorics introduced in \autocite{HKW} in order to deal with approximately uncorrelated entries. It is common practice to analyze the right hand side of \eqref{eq:momenttosum} using the language of graph theory. We begin by summarizing all the graph theoretical notions we need. As in the first author's previous work \autocite{FleermannMT} we follow the expositions \autocite{Stanley, Tittmann}.

\begin{definition}
Let $M$ be a finite set, $k\in\N_0$ be arbitrary, then we denote by $\binom{M}{k}$ the set of all $k$-element subsets of  $M$. A \emph{graph} $G$ is a triple $G=(V,E,\phi)$, where the following holds:
\begin{enumerate}[i)]
	\item $V$ is a finite set, whose elements are called \emph{vertices}, or \emph{nodes}.
	\item $E$ is a finite set, whose elements are called \emph{edges}.
	\item $\phi:E\longrightarrow \binom{V}{1}\cup \binom{V}{2}$ is a function, which is called \emph{incidence function}.
\end{enumerate}
Given arbitrary elements $e\in E$ and $u,v\in V$, such that $\phi(e)=\{u,v\}$, then it is the underlying view that the edge $e$ connects the vertices $u$ and $v$. In this situation, if $u=v$, then $e$ is called \emph{loop}. If $u\neq v$, then $e$ is called \emph{proper edge}. Two different edges $e\neq f\in E$ are called \emph{parallel} if they connect the same nodes, so if $\phi(e) = \phi(f)$. If there are edges $e_1,\ldots,e_k\in E$ which are all parallel to one another, but not parallel to any other edge in $E$, then we call each of the $e_i$ a \emph{$k$-fold edge}. For $k=2$ we use the term \emph{double edge}. If an edge $e$ does not have a parallel edge, $e$ is called a \emph{single edge}. An edge is called $\emph{even}$, if it is a $k$-fold edge with $k$ even, and \emph{odd}, if it is a $k$-fold edge with $k$ odd.
A \emph{path} is a finite sequence of the form
\[
v_1, e_1, v_2, e_2, v_3, e_3, \ldots, v_{k}, e_k, v_{k+1}
\]
for some $k\in\N$, vertices $v_1,\ldots,v_{k+1}\in V$ and edges $e_1,\ldots,e_k\in E$, so that each two neighboring vertices are connected by the edge in between, so $\phi(e_i)=\{v_i,v_{i+1}\}$ for all $i=1,\ldots,k$. If we also have $v_1=v_{k+1}$, then we call the path a \emph{cycle}.
\end{definition}

In order to connect the sum in \eqref{eq:momenttosum} to graphs we introduce  for $\ubar{t}=(t_1,\ldots,t_k)\in\oneto{n}^k$\label{sym:onetoentuple} the notation  $a^{b}_n(\ubar{t})\defeq a^{b}_n(t_1,t_2)  a^{b}_n(t_2,t_3) \cdots a^{b}_n(t_k,t_1)$. To account for the band structure
we call a tuple \emph{$\ubar{t}\in\oneto{n}^k$ $b_n$-relevant}, if each pair $(t_i,t_{i+1})$  (with  $i=1,\ldots,k$ and $k+1\equiv 1$) is $b_n$-relevant (see discussion right before Definition~\ref{def:randommatrix}). Setting
\[
\oneto{n}^k_{b}\defeq\left\{\ubar{t}\in\oneto{n}^k: \ubar{t} \text{ is $b_n$-relevant}\right\}
\]
we obtain
\begin{equation}\label{eq:momenttosumthree}
\integrala{\sigma_n}{x^k}= \frac{1}{nb_n^{k/2}} \sum_{\ubar{t}\in\oneto{n}^k_{b}} a^{b}_n(\ubar{t}).
\end{equation}




The connection to graphs is as follows. Each tuple $\ubar{t}\in\oneto{n}^k$ is identified with the graph $G_{\ubar{t}}=(V_{\ubar{t}},E_{\ubar{t}},\phi_{\ubar{t}})$ with vertices $V_{\ubar{t}} = \{t_1,\ldots,t_k\}$ and (abstract) edges $E_{\ubar{t}}=\{e_1,\ldots, e_k\}$, where $\phi_{\ubar{t}}(e_i)=\{t_i,t_{i+1}\}$  ($i=1,\ldots,k$ and $k+1\equiv 1$), as well as with its cycle
\begin{equation}
\label{eq:eulerian}
t_1,e_1,t_2,e_2,\ldots t_{k-1},e_{k-1},t_k,e_{k},t_1\,.
\end{equation}

Our first combinatorial task is to bound the number of vertices using information on the edges. To this end we introduce for $\ubar{t}\in\oneto{n}^k$ its \emph{profile} $\kappa(\ubar{t}) \in \N_{0}^k$ that records for $1 \leq l \leq k$ how many different $l$-fold edges $\ubar{t}$ contains, and the number $\ell(\ubar{t})$ of different loops in $\ubar{t}$:

\begin{eqnarray*}
\kappa_l(\ubar{t}) &\defeq&
\#\{\phi_{\ubar{t}}(e)\,|\,e\in E_{\ubar{t}} \text{ is an $l$-fold edge}\}, \quad l\in\{1,\ldots,k\},\\
\ell(\ubar{t}) &\defeq& \#\{\phi_{\ubar{t}}(e)\, |\, e \text{ is a loop in } E_{\ubar{t}} \}.
\end{eqnarray*}
To illustrate these definitions consider the tuple $\ubar{t}=(1,1,2,3,2,6,7,6,2,6,2)$. It has one single edge $\{1\}$, three different double edges, $\{1,2\}$, $\{2,3\}$ and $\{6,7\}$, and one $4$-fold edge, $\{2,6\}$. Therefore, $\kappa(\ubar{t}) = (1,3,0,1,0,0,0,0,0,0,0)$ and $\ell(\ubar{t})=1$.

Note that counting the number of edges yields
$k = \sum_{l=1}^{k} l\cdot\kappa_{l}(\ubar{t})$ for each  $\ubar{t}\in\oneto{n}^k$.


\begin{lemma}\label{lem:bound}
Let $n, k \in\N$ and $\ubar{t}\in\oneto{n}^k$ be arbitrary, then
\begin{enumerate}[i)]
	\item $\# V_{\ubar{t}}\leq 1 + \kappa_1(\ubar{t}) + \ldots + \kappa_k(\ubar{t}) - \ell(\ubar{t})$.
	\item If $\ubar{t}$ contains at least one odd edge, then
	\,$\#V_{\ubar{t}}\leq  \kappa_1(\ubar{t}) + \ldots + \kappa_k(\ubar{t}).$
\end{enumerate}
\end{lemma}

\begin{remark}
We also use a weaker version of statement i) of the above lemma by dropping the term $\ell(\ubar{t})$ on the right hand side of the inequality.
\end{remark}

\begin{proof}[Proof of Lemma~\ref{lem:bound}] Let $\ubar{t}\in\oneto{n}^k$ be arbitrary. For the proofs we travel the cycle
\begin{equation}
\label{eq:walk}
t_1,e_1,t_2,e_2,\ldots,t_k,e_k,t_1
\end{equation}
by picking an initial node $t_i$ and then traversing the edges in (increasing) cyclic order until reaching $t_i$ again. Along the path we count the different nodes that we discover. Observe that loops never discover new vertices and that proper $k$-fold edges discover them only at their first passage. \newline
\underline{i)}
Writing $\ell(\ubar{t}) = \ell_1(\ubar{t}) + \ldots + \ell_k(\ubar{t})$, where $\ell_m(\ubar{t})$ denotes the number of different $m$-fold loops in $\ubar{t}$, and starting our tour at the initial vertex $t_1$ we obtain in this way the estimate $\# V_{\ubar{t}}\leq 1 +  \sum_{s=1}^{k} \kappa_s(\ubar{t}) - \ell_s(\ubar{t})$
that yields the desired inequality.\newline
\underline{ii)}
By statement i) it is enough to consider the case $\ell(\ubar{t}) =0$. Reasoning as above it suffices to identify one edge that does not encounter a new node upon first passage. By assumption there exists an odd integer $m$ with $\kappa_m(\ubar{t}) \geq 1$ so that we have $m$ parallel edges $e_{i_1},\ldots,e_{i_m}$. If $m=1$ we start our tour at $t_{i_1 + 1}$ (with the cyclic interpretation if $t_{i_1}=k$). Then the edge $e_{i_1}$ is last on the cycle and cannot discover a new vertex.

For any odd $m \geq 3$ there must be an index $l\in\{1,\ldots,m\}$ such that $e_{i_l}$ and $e_{i_{l+1}}$ are traversed in the same direction, where $i_{m+1}$ cyclicly becomes $i_1$. Then, start the tour at the vertex $t_{i_l + 1}$ (with the cyclic interpretation if $t_{i_l}=k$). Again we have that none of the edges  $e_{i_1},\ldots,e_{i_m}$ of our $m$-fold edge discovers a new node.
\end{proof}

Next we provide an upper bound on the number of $b_n$-relevant tuples $\ubar{t}\in\oneto{n}^k_{b}$ depending on an upper bound on the number of different vertices in $\ubar{t}$.

\begin{lemma}\label{lem:maxnodestuples}
Let $b=(b_n)_n$ be a sequence of $n$-bandwidths. If $k,n \in\N$ are arbitrary and $l\in\{1,\ldots,k\}$, then
\[
\# \{\ubar{t}\in\oneto{n}^k_{b}\, |\, \#V_{\ubar{t}}\leq l\} \leq k^{k}\cdot n b_n^{l-1}
\]
\end{lemma}
\begin{proof}
By the color structure of a $k$-tuple $\ubar{t}$ we understand the information which entries of $\ubar{t}$ agree with each other. For $\ubar{t}\in\{\ubar{t}'\in\oneto{n}^k_{b}\,|\, \#V_{\ubar{t}'}\leq l\}$ each color structure may be encoded by a map $f:\{1,\ldots,k\}\longrightarrow \{1,\ldots,l\}$ so that there are $l^{k}\leq k^k$ possible color structures. Clearly, the number of different tuples $\ubar{t}$ with a given color structure is bounded by $n b_n^{l-1}$ because we have $n$ choices for the first color that appears at $t_1$ and at most $b_n-1$ choices for any color that appears later. This proves the lemma.

\end{proof}

In our analysis of \eqref{eq:momenttosumthree} we partition the summing indices $\ubar{t}$ into equivalence classes. We call two tuples $\ubar{s},\ubar{t}\in\oneto{n}^k_{b}$ \emph{equivalent}, if their profiles agree. The equivalence class generated by a tuple $\ubar{s} \in \oneto{n}^k_{b}$ is therefore given by
\[
\mc{T}(\ubar{s}) \defeq \{\ubar{t}\in\oneto{n}^k_{b}: \kappa(\ubar{t})=\kappa(\ubar{s})\}.
\]
Bounds on the number of equivalence classes and on their respective cardinalities are:
\begin{lemma}\label{lem:howmany}
Let $n,k\in\N$ be fixed.
\begin{enumerate}[i)]
	\item There are at most $(k+1)^k$ equivalence classes in $\oneto{n}^k_{b}$.
	\item Let $\ubar{s}\in\oneto{n}^k_{b}$ be arbitrary, then
	\begin{enumerate}[a)]
		\item $\#\mc{T}(\ubar{s})\leq k^{k} \cdot nb_n^{\kappa_1(\ubar{s}) + \ldots + \kappa_k(\ubar{s})}$.
		\item If $\ubar{s}$ contains at least one odd edge, we have
		$\,\#\mc{T}(\ubar{s}) \leq k^k \cdot nb_n^{ \kappa_1(\ubar{s}) + \ldots + \kappa_k(\ubar{s})-1}$.
\end{enumerate}
\end{enumerate}
\end{lemma}

\begin{proof}
The first statement is a consequence of the fact that each profile is contained in $\{0,\ldots,k\}^k$. The second claim follows immediately from Lemma~\ref{lem:maxnodestuples} and Lemma~\ref{lem:bound}. Note that the improved estimates of Lemma~\ref{lem:bound}~i) cannot be used here because the number of loops are not determined by the profile.
\end{proof}

When dealing with the variance of moments we need to consider pairs of tuples. We now define the corresponding equivalence classes. For $\ubar{s},\ubar{s}'\in\oneto{n}^k_{b}$ we set
\[
\mc{T}(\ubar{s},\ubar{s}')\defeq
\{(\ubar{t},\ubar{t}')\mid \ubar{t},\ubar{t}'\in\oneto{n}^k_{b}, \kappa(\ubar{t}) = \kappa(\ubar{s}), \kappa(\ubar{t}') = \kappa(\ubar{s}')\}
\]
and partition this set into edge disjoint tuple pairs
\[
\mc{T}^d(\ubar{s},\ubar{s}')\defeq
\{(\ubar{t},\ubar{t}') \in \mc{T}(\ubar{s},\ubar{s}') \mid
 \phi_{\ubar{t}}(E_{\ubar{t}})\cap\phi_{\ubar{t}'}(E_{\ubar{t}'})=\emptyset\}
\]
and into tuples pairs with at least one common edge
\[
\mc{T}^c(\ubar{s},\ubar{s}')\defeq
\{(\ubar{t},\ubar{t}')\in \mc{T}(\ubar{s},\ubar{s}') \mid
\phi_{\ubar{t}}(E_{\ubar{t}})\cap\phi_{\ubar{t}'}(E_{\ubar{t}'})\neq\emptyset\}.
\]
We further partition the set $\mc{T}^c(\ubar{s},\ubar{s}')$ into the subsets of equivalent tuples that have exactly $l$ edges in common. So for each $l\in\{1,\ldots,k\}$ we define
\[
\mc{T}^c_l(\ubar{s},\ubar{s}')\defeq
\{(\ubar{t},\ubar{t}')\in \mc{T}^c(\ubar{s},\ubar{s}') \mid
\#[\phi_{\ubar{t}}(E_{\ubar{t}})\cap\phi_{\ubar{t}'}(E_{\ubar{t}'})]=l\}.
\]

We are now interested in bounds for $\#\mc{T}^d(\ubar{s},\ubar{s}')$, $\# \mc{T}^c(\ubar{s},\ubar{s}')$ and $\# \mc{T}^c_l(\ubar{s},\ubar{s}')$. The first quantity can be trivially bounded by
\begin{equation}\label{eq:disjointbound}
\#\mc{T}^d(\ubar{s},\ubar{s}') \leq \#\mc{T}(\ubar{s},\ubar{s}')=\,\#\mc{T}(\ubar{s})\cdot \#\mc{T}(\ubar{s}').
\end{equation}
To bound the latter two quantities we use the common edge to join the paths $\ubar{t}$ and $\ubar{t}'$ in a specific way. Here we use the notion of cyclic permutations of tupels that we define by example. The cyclic permutations of $(2,8,6,3)$ are given by $(3,2,8,6)$, $(6,3,2,8)$, and $(8,6,3,2)$.

\begin{lemma}\label{lem:graphsuperposition}
Let $\ubar{t}$ and $\ubar{t}'$ in $\oneto{n}^k$ have a common edge. Then there exists $\ubar{u}\in\oneto{n}^{2k}$ with
\begin{enumerate}[i)]
	\item  $(u_1,\ldots,u_k)$ is a cyclic permutation of $(t_1,\ldots,t_k)$ and $(u_{k+1},\ldots,u_{2k})$ is a cyclic permutation of $(t_1',\ldots,t_k')$.
	\item $((u_1,u_2),\ldots,(u_k,u_{k+1}))$ is a cyclic permutation of $((t_1,t_2),\ldots,(t_k,t_1))$ and \newline $((u_{k+1},u_{k+2}),\ldots,(u_{2k},u_1))$ is a cyclic permutation of $((t_1',t_2'),\ldots,(t_k',t_1'))$.
\end{enumerate}
In particular, the cycle $\ubar{u}$ spans the graph obtained through superposition of the graphs of $\ubar{t}$ and $\ubar{t}'$. It travels first through all the edges of $\ubar{t}$ and then through all the edges of $\ubar{t}'$.	
\end{lemma}
\begin{proof}
Let $\ubar{t}$ and $\ubar{t}'$ be as in the statement of the lemma. Then they have a common node. Therefore, there exist cyclic permutations $\ubar{\tilde{t}}$ of $\ubar{t}$ and $\ubar{\tilde{t}'}$ of $\ubar{t'}$ such that $\tilde{t}_1  = \tilde{t}'_1$. Set $\ubar{u}\defeq (\tilde{t}_1,\ldots,\tilde{t}_k,\tilde{t}'_1,\ldots,\tilde{t}'_k)$. It is straight forward to verify claims i) and ii).
\end{proof}

The following two lemmas provide our estimates on $\# \mc{T}^c(\ubar{s},\ubar{s}')$ and $\# \mc{T}^c_l(\ubar{s},\ubar{s}')$.
\begin{lemma}\label{lem:doublebound}
Let $b=(b_n)_n$ be a sequence of $n$-bandwidths and $n,k\in\N$ be fixed. Let $\ubar{s}, \ubar{s}'\in\oneto{n}^k_{b}$. Then the following statements hold:
\begin{enumerate}
\item If $\ubar{s}$ and $\ubar{s}'$ have only even edges, we have for each $(\ubar{t},\ubar{t}')\in \mc{T}^c(\ubar{s},\ubar{s}')$ that
\[
\#(V_{\ubar{t}}\cup V_{\ubar{t}'}) \leq k.
\]
\item Assume that one of the tuples $\ubar{s}$ or $\ubar{s}'$ contains at least one odd edge and let $l\in\{1,\ldots,k\}$. Then for each $(\ubar{t},\ubar{t}')\in \mc{T}^c_l(\ubar{s},\ubar{s}')$ we have
\[
\# (V_{\ubar{t}} \cup V_{\ubar{t}'}) \leq \sum_{j=1}^k \kappa_j(\ubar{s}) + \sum_{j=1}^k \kappa_j(\ubar{s}') - l.
\]
\end{enumerate}
\end{lemma}
\begin{proof}
1. Let $(\ubar{t},\ubar{t}')\in\mc{T}^c(\ubar{s},\ubar{s}')$. Then both $\ubar{t}$ and $\ubar{t}'$ each span at most $\frac{k}{2}+1$ nodes by Lemma~\ref{lem:bound}. The only case that needs some thought is if at least one of them, say $\ubar{t}$, contains exactly $\frac{k}{2}+1$ different vertices. Then $\ubar{t}$ consists of double edges only, all of them \emph{proper}. In particular, since having a proper edge with $\ubar{t}$ in common, $\ubar{t}'$ can span at most $\frac{k}{2}+1-2$ additional nodes, leading to a total of at most
$k$ different nodes.

2. We begin with the case that $\ubar{s}$ contains an odd edge. Since $\ubar{t}$ and $\ubar{t}'$ have at least one edge in common we may define $\ubar{u}\in\oneto{n}^{2k}$ as in Lemma~\ref{lem:graphsuperposition}. Obviously, $\#(V_{\ubar{t}}\cup V_{\ubar{t}'}) = \# V_{\ubar{u}}$. Let us travel along the cycle $\ubar{u}$ and observe how many different nodes we discover. By travelling the first $k$ edges of $\ubar{u}$ we actually travel the edges of $\ubar{t}$ by Lemma~\ref{lem:graphsuperposition}~ii) and may at most discover $\kappa_1(\ubar{t})+\ldots +\kappa_k(\ubar{t})$ nodes by Lemma~\ref{lem:bound}~ii). Travelling the remaining $k$ edges of $\ubar{u}$ we may subtract $1+l$ from the general upper bound $1+\kappa_1(\ubar{t}')+\ldots +\kappa_k(\ubar{t}')$ (cf.~proof of  Lemma~\ref{lem:bound}) because the first node is not new and because all first passages along edges that are common edges with $\ubar{t}$ (there are $l$ of them) do not discover a new node. This proves the second claim. In case $\ubar{s}$ does not contain an odd edge, $\ubar{s}'$ does. Interchanging the roles of $\ubar{t}$ and $\ubar{t}'$ allows to use the same argument for the proof.
\end{proof}

\begin{lemma}\label{lem:doublehowmany}
Let $b=(b_n)_n$ be a sequence of $n$-bandwidths. Fix $n,k\in\N$ and let $\ubar{s}$ and $\ubar{s}'$ in $\oneto{n}^k_{b}$ be arbitrary, then
	\begin{enumerate}[a)]
		\item If both $\ubar{s}$ and $\ubar{s}'$ contain only even edges, then $\#\mc{T}^c(\ubar{s},\ubar{s}')\leq k^2\cdot (2k)^{2k} \cdot nb_n^{k-1}$.
		\item If $\ubar{s}$ or $\ubar{s}'$ contains at least one odd edge, we have
		\[
		\#\mc{T}^c(\ubar{s},\ubar{s}') \leq k^2 \cdot (2k)^{2k} \cdot nb_n^{\kappa_1(\ubar{s}) + \ldots + \kappa_k(\ubar{s})+\kappa_1(\ubar{s}') + \ldots + \kappa_k(\ubar{s}')-2}.
		\]
		\item If $\ubar{s}$ or $\ubar{s}'$ contains at least one odd edge, we have for all $l\in\{1,\ldots,k\}$, that
		\[
		\#\mc{T}^c_l(\ubar{s},\ubar{s}') \leq k^2 \cdot (2k)^{2k} \cdot nb_n^{\kappa_1(\ubar{s}) + \ldots + \kappa_k(\ubar{s})+\kappa_1(\ubar{s}') + \ldots + \kappa_k(\ubar{s}')-l-1}.
		\]
\end{enumerate}
\end{lemma}

\begin{proof}
a) Pick $\ubar{s}$ and $\ubar{s}'$ in $\oneto{n}^k_{b}$ with only even edges. In order to bound the number of tuple pairs $(\ubar{t},\ubar{t}')\in\mc{T}^c(\ubar{s},\ubar{s}')$ we first estimate the number of cycles $\ubar{u}\in\oneto{n}^{2k}_{b}$ that can be constructed from $\ubar{t}$ and $\ubar{t}'$ via Lemma~\ref{lem:graphsuperposition}. Statement 1.~of Lemma~\ref{lem:doublebound} together with Lemma~\ref{lem:maxnodestuples} yield the upper bound $(2k)^{2k}\cdot n b_n^{k-1}$. Statement a) follows since each of these cycles $\ubar{u}\in\oneto{n}^{2k}_{b}$ can be generated by at most $k^2$ different  tuple pairs $(\ubar{t},\ubar{t}')\in\mc{T}^c(\ubar{s},\ubar{s}')$.

Claims b) and c) are proved in a similar way, employing statement 2.~of Lemma~\ref{lem:doublebound}.

\end{proof}


\subsection{Convergence of Expected Moments}

In this subsection we show relation~(\ref{eq:toshowfirst}) of Theorem~\ref{thm:expectationandvariance} and consequently the semicircle law in expectation for the ensembles considered in our main Theorem~\ref{thm:main}. The proof is a combination of arguments presented in \autocite{Bogachev} and \autocite{HKW}. We include the proof here to keep the paper self-contained and as an opportunity for the reader to familiarize herself or himself with the language and method of proof in a situation that is less involved than in the next subsection.
\begin{theorem}\label{thm:expectationtocatalan}
Let $(X_n)_n$ be a sequence of periodic random band matrices which is based on an $\alpha$-almost uncorrelated triangular array $(a_n)_n$ with $\alpha\geq\frac{1}{2}$ and bandwidth $b=(b_n)_n$. Then if $b_n\to\infty$, we have for all $k\in\N$, that
\[
\E\integrala{\sigma_n}{x^k} \xrightarrow[n\to\infty]{} \integrala{\sigma}{x^k}.
\]
\end{theorem}
\begin{proof}
Keeping relation~\eqref{eq:momenttosumthree} in mind the basic approach to the proof is to sum $\E a^{b}_n(\ubar{t})$ over equivalence classes $t \in \mc{T}(\ubar{s})$ and to perform the $n$-limit for these partial sums. This suffices as the number of equivalence classes is bounded by $(k+1)^k$ (see Lemma~\ref{lem:howmany}).

\underline{Step 1: Let $k\in\N$ be odd.}\newline
For $\ubar{s}\in\oneto{n}^k_{b}$  condition (AAU1) implies $\abs{\E a_n^{b}(\ubar{t})} \leq C_k
/n^{\alpha\cdot\kappa_1(\ubar{s})}\leq C_k
/ b_n^{\alpha\cdot \kappa_1(\ubar{s})}$ for all $\ubar{t} \in\mc{T}(\ubar{s})$. As $\ubar{s}$  must have at least one odd edge, Lemma~\ref{lem:howmany} yields
\begin{align*}
\bigabs{\frac{1}{nb_n^{k/2}}\sum_{\ubar{t}\in \mc{T}(\ubar{s})}\hspace{-0.2cm}{\E a_n^{b}(\ubar{t})}}
&\leq \frac{C_k
}{nb_n^{k/2}}\cdot k^k \cdot nb_n^{(1-\alpha)\cdot\kappa_1(\ubar{s})+\kappa_2(\ubar{s}) + \ldots +\kappa_{k}(\ubar{s}) -1}.
\end{align*}
For $\alpha\geq\frac{1}{2}$, the last exponent is maximal if $\ubar{s}$ consists of one single edge and double edges otherwise due to the constraint $k = \sum_{l=1}^{k} l\cdot\kappa_{l}(\ubar{t})$ (see line above Lemma~\ref{lem:bound}). Thus the exponent is at most $(1-\alpha) + \frac{k-1}{2} -1\leq\frac{k}{2}-1$ and we have
\begin{equation}\label{eq:oddestimateT}
\bigabs{\frac{1}{nb_n^{k/2}}\sum_{\ubar{t}\in \mc{T}(\ubar{s})}\hspace{-0.2cm}{\E a_n^{b}(\ubar{t})}}
\leq \frac{C_k}
{b_n^{k/2}} \cdot k^k \cdot b_n^{k/2-1} \xrightarrow[n\to\infty]{} 0.
\end{equation}
Summation over all relevant equivalence classes shows $\E\integrala{\sigma_n}{x^k} \xrightarrow[n\to\infty]{}0$ as required by relation~\eqref{eq:momentsofsemicircle}.
\underline{Step 2: Let $k\in\N$ be even.}\newline
We distinguish three different types of equivalence classes.

\underline{Case 1:} Let $\ubar{s}\in\oneto{n}^k_{b}$, so that $\ubar{s}$ contains an odd edge.
We can derive statement~\eqref{eq:oddestimateT} with the same line of reasoning as in Step 1, except for the maximization argument for the exponent. Now, as $k$ is even, the exponent is maximal if $\ubar{s}$ consists of double edges only. Nevertheless the exponent is still bounded above by $\frac{k}{2}-1$.


\underline{Case 2:}
Let $\ubar{s}\in\oneto{n}^k$ have only even edges, but at least one $m$-fold edge with $m\geq 4$.
In this case Lemma~\ref{lem:howmany} implies
\[
\# \mc{T}(\ubar{s}) \leq k^k \cdot nb_n^{\kappa_{2}(\ubar{s})+\kappa_{4}(\ubar{s})\ldots + \kappa_{k}(\ubar{s})}.
\]
The exponent is maximized when $\ubar{s}$ has one $4$-fold edge and just double edges otherwise, and then we obtain the exponent $1 + (k-4)/2 = k/2-1$.
Using $\abs{\E a_n^{b}(\ubar{t})} \leq  C_k
$ from $(AAU1)$ we again may conclude statement~\eqref{eq:oddestimateT}.

\underline{Case 3:}
Let $\ubar{s}\in\oneto{n}^k_{b}$ consist of double edges only and partition the set $\mc{T}(\ubar{s})$ into sets
\[
\mc{T}_{\frac{k}{2}+1}(\ubar{s})\defeq \left\{\ubar{t}\in\mc{T}(\ubar{s}): \# V_{\ubar{t}}=\frac{k}{2}+1\right\}
\qquad \text{and} \qquad
\mc{T}_{\leq \frac{k}{2}}(\ubar{s}) \defeq \left\{\ubar{t}\in\mc{T}(\ubar{s}): \# V_{\ubar{t}} \leq \frac{k}{2}\right\}.
\]
By Lemma~\ref{lem:maxnodestuples} we have
\[
\#\mc{T}_{\leq \frac{k}{2}}(\ubar{s}) \leq k^k n b_n^{\frac{k}{2}-1}.
\]
Again condition (AAU1) leads to the same estimate as in line~\eqref{eq:oddestimateT} for the partial sum over $\ubar{t}\in \mc{T}_{\leq\frac{k}{2}}(\ubar{s})$. Finally, we turn to those summands that contribute in the limit. A version of the following argument was already at the heart of Wigner's proof in \autocite{Wigner}.
To count the possibilities to construct a $\ubar{t}\in\mc{T}_{\frac{k}{2}+1}(\ubar{s})$, we first pick an appropriate coloring  $f:\{1,\ldots,k\}\longrightarrow \{1,\ldots,k/2+1\}$ in \emph{standard form}. This means $f(1)=1$ and for $l>1$ we have that if $f(l)\neq f(j)$ for all $j<l$, then $f(l)=\max\{f(j):\,j<l\}+1$. The possible standard colorings for $k$-tuples with only proper double edges and $k/2+1$ different vertices are in bijective correspondence to Dyck paths of length $k$, and there are exactly $\Cat_{\frac{k}{2}}$ of them. For example, the tuple $(8,5,6,9,6,2,6,5)$ has the standard coloring scheme  $(1,2,3,4,3,5,3,2)$, which is associated with the difference sequence of the Dyck path $(1,1,1,-1,1,-1,-1,-1)$. For a formal proof of this we refer the reader to \autocite[15]{Anderson}: There, we note that given a coloring $f$ as above, we obtain the associated Wigner word representative $(f(1),f(2),\ldots,f(k),f(1))$ (and vice versa) as in the proof of their Lemma 2.1.6.

Given such a standard coloring $f$, to construct a $b_n$-relevant tuple in $\oneto{n}^k$ matching this coloring we have at least $n \cdot (b_n-1) \cdots (b_n-\frac{k}{2})$ and at most $n\cdot b_n^{\frac{k}{2}}$ possibilities. Thus
\[
\frac{1}{n b_n^{k/2}}\#\mc{T}_{\frac{k}{2}+1}(\ubar{s}) \xrightarrow[n\to\infty]{} \Cat_{\frac{k}{2}}.
\]
In addition, the second moment property (AAU2) implies for all $\ubar{t}\in\mc{T}(\ubar{s})$ that
\[
\abs{\E a_n(\ubar{t})-1} \leq C^{(k/2)}_n
\]
where $C^{(k/2)}_n$ converges to $0$ as $n\to\infty$. Combining the last two statements
we obtain
\[
\bigabs{\Cat_{\frac{k}{2}}-\frac{1}{n b_n^{k/2}}\sum_{\ubar{t}\in\mc{T}_{\frac{k}{2}+1}(\ubar{s})}\hspace{-0.2cm}{\E a_n^{b}(\ubar{t})} }
\leq
\bigabs{\Cat_{\frac{k}{2}}-  \frac{\#\mc{T}_{\frac{k}{2}+1}(\ubar{s})}{n b_n^{k/2}}}   +
C^{(k/2)}_n   \frac{\#\mc{T}_{\frac{k}{2}+1}(\ubar{s})}{n b_n^{k/2}}
 \xrightarrow[n\to\infty]{}0.
\]

\underline{Conclusion of Step 2:}
Adding all contributions from the three types of equivalence classes we arrive at $\E\integrala{\sigma_n}{x^k} \xrightarrow[n\to\infty]{}\Cat_{\frac{k}{2}}$ in accordance with relation~\eqref{eq:momentsofsemicircle}.
\end{proof}


\subsection{Decay of Variance of Moments}

Finally, we analyze the variances of the moments of the empirical spectral distributions and their limiting bahavior as the matrix dimension tends to infinity.

\begin{theorem}\label{thm:variancetozero}
Let $(X_n)_n$ be a sequence of periodic random band matrices which is based on an $\alpha$-almost uncorrelated triangular array $(a_n)_n$ with $\alpha\geq\frac{1}{2}$ and bandwidth $b=(b_n)_n$. Denote by $(\sigma_n)_n$ the ESDs of $(X_n)_n$. Then the following holds:

\begin{itemize}
\item[i)] If $b_n\to\infty$, then
$
\V\integrala{\sigma_n}{x^k} \xrightarrow[n\to\infty]{} 0
$ for all $k\in\N$.
\end{itemize}
If, in addition, at least one of the conditions ii)-iv) is satisfied then
\[
\V\integrala{\sigma_n}{x^k} \xrightarrow[n\to\infty]{} 0 \quad \text{summably fast for all $k\in\N$}.
\]
\begin{itemize}
\item[ii)] All random variables of $(a_n)_n$ are $\{+1,-1\}$-valued and $\frac{1}{b_n^3}$ is summable over $n$.
\item[iii)] $(a_n)_n$ is strongly $\alpha$-almost uncorrelated with  $\alpha>\frac{1}{2}$, and the sequences $\frac{1}{b_n^2}$, $\frac{1}{b_n}D^{(l)}_n$ and $C^{(l)}_n$ are summable over $n$ for all $l\in\N$ .
\item[iv)] $(a_n)_{n\in\N}$ is a Wigner scheme and $\frac{1}{nb_n}$ is summable over $n$.
\end{itemize}
\end{theorem}


Before we begin with the proof of Theorem~\ref{thm:variancetozero}, let us formulate a lemma which facilitates the use of condition (AAU3).


\begin{lemma}\label{lem:useaauthree}
Let $(a_n)_n$ be a strongly $\alpha$-almost uncorrelated triangular scheme. Then for all $l,N\in\N$, $l\geq 3$ odd, and fundamentally different pairs $P_1,\ldots,P_l$ in $\sqr{N}$ we have for all $n\geq N$:
\begin{align*}
&\bigabs{\E \left [a_n(P_1)^4 a_n(P_2)^2\cdots a_n(P_l)^2\right ]
\,-\, \E \left [ a_n(P_1)^4 a_n(P_2)^2\cdots a_n(P_{l_1})^2\right ] \cdot\E\left  [ a_n(P_{l_2})^2\cdots a_n(P_l)^2\right ]} \\
&\leq D^{(l)}_n + 
C_4
\cdot C^{(l_2)}_n + D^{(l_1)}_n\cdot 
C_{2l_2}
,
\end{align*}
where we set $l_1\defeq \frac{l-1}{2}$ and $l_2\defeq\frac{l+1}{2}$ (thus $l-l_2=l_1$). 
\end{lemma}
\begin{proof}
Adding and subtracting $\E a_n(P_1)^4-\E a_n(P_1)^4\cdot \E a_n(P_{l_2})^2 \cdots a_n(P_l)^2$ we obtain
\begin{align*}
&\bigabs{\E\left[ a_n(P_1)^4 a_n(P_2)^2\cdots a_n(P_l)^2\right]\,-\, \E \left[a_n(P_1)^4 a_n(P_2)^2\cdots a_n(P_{l_1})^2\right]\cdot\E\left[ a_n(P_{l_2})^2 \cdots a_n(P_l)^2\right]}\\
&= | \E \left[a_n(P_1)^4[a_n(P_2)^2\cdots a_n(P_l)^2-1]\right] \,+\, \E\left[ a_n(P_1)^4\right] \\
&\quad - \, \E\left[ a_n(P_1)^4[a_n(P_2)^2\cdots a_n(P_{l_1})^2-1]\right]\cdot \E \left[a_n(P_{l_2})^2 \cdots a_n(P_l)^2\right] \\
&\quad -\,  \E\left[ a_n(P_1)^4\right]\cdot \E\left[ a_n(P_{l_2})^2 \cdots a_n(P_l)^2\right] |\\
&\leq D^{(l)}_n\, +\, \abs{\E \left[a_n(P_1)^4\right]}\abs{\E \left[a_n(P_{l_2})^2 \cdots a_n(P_l)^2-1\right]} \, + \,D^{(l_1)}_n\cdot 
C_{2l_2}
\\
&\leq D^{(l)}_n \,+ 
C_4
\cdot C^{(l_2)}_n \,+\, D^{(l_1)}_n\cdot 
C_{2l_2}
.
\end{align*}
\end{proof}

\begin{proof}[Proof of Theorem~\ref{thm:variancetozero}]
Since
$\V\integrala{\sigma_n}{x^k} = \E\left(\integrala{\sigma_n}{x^k}^2\right) - \left(\E\integrala{\sigma_n}{x^k}\right)^2$
it suffices to show
\begin{equation}\label{eq:star}
\frac{1}{n^2b_n^k} \sum_{\ubar{t},\ubar{t}'\in\oneto{n}_{b}^k} \left( \E a_n(\ubar{t}) a_n(\ubar{t}') - \E a_n(\ubar{t}) \E a_n(\ubar{t}') \right) \xrightarrow[n\to\infty]{} 0
\end{equation}
and to determine when this convergence is summably fast.

To analyze the sum in \eqref{eq:star} we proceed in a similar fashion as in the proof of Theorem~\ref{thm:expectationtocatalan}. We identify different cases of profiles for $\ubar{s}$ and $\ubar{s}'$ such that bounds on the partial sums over $(\ubar{t},\ubar{t}')$ in the subsets $\mc{T}^d(\ubar{s},\ubar{s}')$, $\mc{T}^c(\ubar{s},\ubar{s}')$ of the corresponding equivalence classes $\mc{T}(\ubar{s},\ubar{s}')$ (see discussion below Lemma~\ref{lem:howmany}) can be derived by the same reasoning. Since the total number of equivalence classes is bounded by the $n$-independent number $(k+1)^{2k}$ (cf.~Lemma~\ref{lem:howmany}~i) the bounds on the partial sums suffice to prove the theorem.

In each of the following cases, we determine which conditions in addition to $(a_n)_n$ being an \emph{$\alpha$-almost uncorrelated triangular array with $\alpha \geq 1/2$} are needed for regular convergence (i.e.~convergence per se) and for summable convergence. Our findings are summarized at the beginning of each case. We close by discussing  the statements i) through iv) of Theorem~\ref{thm:variancetozero} separately.

\noindent\underline{1. Step: Disjoint edge sets}\newline
In view of \eqref{eq:star} we need to derive upper bounds on
\[
P_d (\ubar{s},\ubar{s}') \defeq
\bigabs{\frac{1}{n^2b_n^k} \sum_{(\ubar{t},\ubar{t}')\in\mc{T}^d(\ubar{s},\ubar{s}')} \left( \E a_n^{b}(\ubar{t}) a_n^{b}(\ubar{t}') - \E a_n^{b}(\ubar{t}) \E a_n^{b}(\ubar{t}') \right)}
\]


\noindent\underline{1. Case: Both $\ubar{s}$ and $\ubar{s}'$ have only even edges.}\newline
\noindent\underline{1. Subcase: $\kappa_2(\ubar{s})= k/2 = \kappa_2(\ubar{s}')$}\newline
[Claim: We achieve regular convergence if the sequences $(C^{(l)}_n)_n$ converge to zero and summable convergence if the sequences $(C^{(l)}_n)_n$ are summable.]
Due to conditions (AAU1) and (AAU2) we have for all $(\ubar{t},\ubar{t}')\in\mc{T}^d(\ubar{s},\ubar{s}')$
\begin{multline*}
\abs{\E a_n^{b}(\ubar{t}) a_n^{b}(\ubar{t}') - \E a_n^{b}(\ubar{t}) \E a_n^{b}(\ubar{t}')}\\
\leq \abs{\E a_n^{b}(\ubar{t}) a_n^{b}(\ubar{t}')-1} + \abs{\E a_n^{b}(\ubar{t}) -1} \cdot \abs{\E a_n^{b}(\ubar{t}')} + \abs{\E a_n^{b}(\ubar{t}') -1} \leq D_n
\end{multline*}
with $D_n\defeq C^{(k)}_n + C^{(k/2)}_n
\cdot C_k
+ C^{(k/2)}_n$. Moreover, Lemma~\ref{lem:howmany} and \eqref{eq:disjointbound} imply $\#\mc{T}^d(\ubar{s},\ubar{s}') \leq k^{2k} \cdot n^2b_n^k$. Hence $P_d (\ubar{s},\ubar{s}') \leq k^{2k} D_n$ and the claim follows.


\noindent
\underline{2. Subcase: We have $\kappa_l(\ubar{s})\geq 1$ or $\kappa_l(\ubar{s}')\geq 1$ for some even $l\geq 4$}.
\newline
[Claim: We achieve regular convergence if $b_n\to\infty$, and summable convergence if (AAU3) holds and $1/b_n^2$, $D^{(l)}_n\cdot 1/b_n$,  $C^{(l)}_n$ are summable for all $l$.]
Condition (AAU1) provides the following basic bound on $\abs{\E a_n^{b}(\ubar{t}) a_n^{b}(\ubar{t}') - \E a_n^{b}(\ubar{t}) \E a_n^{b} (\ubar{t}')}$ for all $(\ubar{t},\ubar{t}')\in\mc{T}(\ubar{s},\ubar{s}')$
\begin{eqnarray}\label{summandsimpleestimate}
\abs{\E a_n^{b}(\ubar{t}) a_n^{b}(\ubar{t}')} + \abs{\E a_n^{b}(\ubar{t})}\cdot \abs{\E a_n^{b}(\ubar{t}')} \leq 
C_{2k} + C_k^2
=:B.\qquad
\end{eqnarray}
In view of Lemma~\ref{lem:howmany} the following observation is useful. For $\ubar{s} \in \oneto{n}^k_{b}$ with only even edges  the sum of the entries of its profil satisfies
\[
\kappa_2(\ubar{s}) + \kappa_4(\ubar{s}) + \ldots + \kappa_k(\ubar{s}) \quad
\begin{cases}=k/2,&\text{if } \kappa_2(\ubar{s})= k/2,\\=k/2-1,&\text{if } \kappa_2(\ubar{s})= k/2-2 \text{ and } \kappa_4(\ubar{s})=1,\\ \leq k/2-2,&\text{else}.
\end{cases}
\]
Thus $P_d (\ubar{s},\ubar{s}') \leq k^{2k}\cdot B \cdot b_n^{-1}$ holds always in this subcase proving the claim on regular convergence.

Observe that the bound $P_d (\ubar{s},\ubar{s}') \leq k^{2k}\cdot B \cdot b_n^{-2}$ would prove the remaining claim on summable convergence. By the reasoning above the only cases where this bound does \emph{not} hold is if $\kappa_4(\ubar{s})=1$,
$\kappa_2(\ubar{s})=(k-4)/2$ and $\kappa_2(\ubar{s}')=k/2$ or vice versa (that is, we have one 4-fold edge and double edges otherwise). Then we must resort to (AAU3) and Lemma~\ref{lem:useaauthree} (note that $a_n^b(\ubar{t})=a_n(\ubar{t})$ for $b_n$-relevant tuples $\ubar{t}$)
with $l=k-1$. All in all we obtain in these two exceptional cases the estimate
\[
P_d (\ubar{s},\ubar{s}') \leq k^{2k}\cdot \left(  D^{(k-1)}_n + 
C_4
\cdot C^{(k/2)}_n + D^{(k/2-1)}_n\cdot 
C_k
\right) \cdot b_n^{-1}
\]
which converges summably fast if $\frac{1}{b_n}D^{(l)}_n$ and $C^{(l)}_n$ are summable for all $l$. This completes the discussion of Case 1.~in which $\ubar{s}$ and $\ubar{s}'$ have only even edges.

\noindent\underline{2. Case: $\ubar{s}$ has at least one odd edge and $\ubar{s}'$ has only even edges, or vice versa}.\newline
Throughout this case, we assume that $\ubar{s}$ has at least one odd edge and $\ubar{s}'$ has only even edges. Then, actually, $\ubar{s}$ has at least two odd edges, since the total number of edges must be even. Moreover, it suffices to use the following estimate on the summands that immediately follows for all $(\ubar{t},\ubar{t}')\in\mc{T}(\ubar{s},\ubar{s}')$ from condition (AAU1) and from the definition of $B$ in \eqref{summandsimpleestimate}
\begin{equation}\label{eq:case2summands}
\abs{\E a_n^{b}(\ubar{t}) a_n^{b}(\ubar{t}') - \E a_n^{b}(\ubar{t}) \E a_n^{b} (\ubar{t}')} \leq B n^{-\alpha\cdot\kappa_1(\ubar{s})} .
\end{equation}
Observe in addition that the number of summands in $P_d (\ubar{s},\ubar{s}')$ is bounded by
\begin{equation}\label{eq:case2numbersummands}
\#\mc{T}^d(\ubar{s},\ubar{s}') \leq k^{2k} \cdot n^2 b_n^{K(\ubar{s})+K(\ubar{s}')-1}\quad\text{with}\quad K(\ubar{s}) \defeq \kappa_1(\ubar{s}) + \kappa_2(\ubar{s}) + \ldots + \kappa_{k}(\ubar{s})
\end{equation}
due to Lemma~\ref{lem:howmany} and \eqref{eq:disjointbound}.

\noindent\underline{1. Subcase:  $\ubar{s}$ has no $m$-fold edge with $m\geq 3$}
\newline
[Claim: We always have regular convergence. Summable convergence holds for $\alpha >1/2$.]
In this subcase we have $K(\ubar{s}')\leq k/2$ and $K(\ubar{s})=\kappa_1(\ubar{s})+(k-\kappa_1(\ubar{s}))/2$ with $\kappa_1(\ubar{s})\geq 2$ even. Combining the estimates \eqref{eq:case2summands} and \eqref{eq:case2numbersummands} yields for $\alpha\geq 1/2$
\[
P_d (\ubar{s},\ubar{s}') \leq k^{2k}\cdot B \cdot b_n^{\kappa_1(\ubar{s})/2-1} n^{-\alpha\cdot\kappa_1(\ubar{s})} \leq k^{2k}\cdot B \cdot b_n^{2/2-1} n^{-\alpha \cdot2} = k^{2k}\cdot B \cdot n^{-2\alpha}
\]
proving the claim.

\noindent\underline{2. Subcase: $\ubar{s}$ has an $m$-fold edge, $m\geq 3$.}\newline
[Claim: We achieve regular convergence if $b_n \to\infty$ and summable convergence if $1/b_n^2$ is summable and $\alpha > 1/2$.] Again we use $K(\ubar{s}')\leq k/2$. In order to bound $K(\ubar{s})$ we distinguish the cases of even and odd values for $\kappa_1(\ubar{s})$.

If $\kappa_1(\ubar{s})$ is odd then $K(\ubar{s})$ is maximal if there is exactly one $3$-fold edge and all remaining edges are double edges leading to
\[
K(\ubar{s})\leq\kappa_1(\ubar{s})+\frac{k-\kappa_1(\ubar{s})-3}{2}+1 .
\]
For $\alpha\geq 1/2$ we obtain with estimates \eqref{eq:case2summands} and \eqref{eq:case2numbersummands}
\[
P_d (\ubar{s},\ubar{s}') \leq k^{2k}\cdot B \cdot b_n^{(\kappa_1(\ubar{s})-3)/2} n^{-\alpha\cdot\kappa_1(\ubar{s})} \leq k^{2k}\cdot B \cdot b_n^{(1-3)/2} n^{-\alpha \cdot 1} = k^{2k}\cdot B \cdot b_n^{-1} n^{-\alpha}.
\]
Thus regular convergence is established and convergence in summability follows via the Cauchy-Schwarz inequality from the summability of $1/b_n^2$ and of $n^{-2\alpha}$.

If $\kappa_1(\ubar{s})$ is even then there exist at least two odd edges that are each at least $3$-fold. Consequently,
\[
K(\ubar{s})\leq\kappa_1(\ubar{s})+\frac{k-\kappa_1(\ubar{s})-6}{2}+2 = \frac{k+\kappa_1(\ubar{s})}{2}-1
\]
and we obtain for $\alpha\geq 1/2$
\[
P_d (\ubar{s},\ubar{s}') \leq k^{2k}\cdot B \cdot b_n^{\kappa_1(\ubar{s})/2-2} n^{-\alpha\cdot\kappa_1(\ubar{s})} \leq k^{2k}\cdot B \cdot b_n^{0/2-2} n^{-\alpha \cdot 0} = k^{2k}\cdot B \cdot b_n^{-2}
\]
from the estimates \eqref{eq:case2summands} and \eqref{eq:case2numbersummands}. This completes the proof of the claim and closes our discussion of Case 2.

\noindent\underline{3. Case: Both $\ubar{s}$ and $\ubar{s}'$ have at least one odd edge.}\newline
[Claim: We achieve regular convergence if $b_n\to\infty$ and summable convergence if $1/b_n^3$ is summable.]
In this case application of  (AAU1), Lemma~\ref{lem:howmany}, and \eqref{eq:disjointbound} gives
\begin{equation}\label{eq:basicestimatecase3}
P_d (\ubar{s},\ubar{s}') \leq k^{2k}\cdot B \cdot \big( b_n^{K(\ubar{s})-1-k/2}\, n^{-\alpha\cdot\kappa_1(\ubar{s})} \big) \cdot \big( b_n^{K(\ubar{s}')-1-k/2}\, n^{-\alpha\cdot\kappa_1(\ubar{s}')} \big)
\end{equation}

Observe that in this case $k$ may be odd.
To bound the last two factors in \eqref{eq:basicestimatecase3} we now justify the estimate
\begin{equation}
\label{eq:casethreeexplanation}
\frac{b_n^{K(\ubar{s})-1-k/2}}{n^{\alpha\cdot\kappa_1(\ubar{s})}} \leq \max\left(\frac{1}{b_n^{3/2}},\frac{1}{b_n^{1/2}n^{\alpha}}\right)
\end{equation}
The first bound in the maximum applies to the case $\kappa_1(\ubar{s})=0$. In this case, one has $K(\ubar{s})\leq\lfloor(k-1)/2 \rfloor$ because there has to be at least one odd edge. The second bound in the maximum applies to the case $\kappa_1(\ubar{s})>0$. In this case, $K(\ubar{s})\leq \kappa_1(\ubar{s})+\lfloor (k-\kappa_1(\ubar{s}))/2\rfloor$, so the l.h.s.\ of \eqref{eq:casethreeexplanation} is bounded by $b_n^{\kappa_1(\ubar{s})/2-1}/n^{\alpha\kappa_1(\ubar{s})}$, which is maximal at $\kappa_1(\ubar{s})=1$, since $\alpha\geq 1/2$. The same bound \eqref{eq:casethreeexplanation} applies to the last factor in \eqref{eq:basicestimatecase3}, and regular convergence follows. It remains to argue the summability of $b_n^{-3}$, $b_n^{-2}n^{-\alpha}$, and $b_n^{-1}n^{-2\alpha}$. The first sequence is summable by assumption, the second using the Hölder estimate $\norm{b_n^{-2}n^{-\alpha}}_1\leq \norm{b_n^{-2}}_{3/2}\norm{n^{-\alpha}}_3$, the third using $\norm{b_n^{-1}n^{-2\alpha}}_1\leq \norm{b_n^{-1}}_{3}\norm{n^{-2\alpha}}_{3/2}$.

\noindent\underline{2. Step: Non-disjoint edge sets}\\
In this second step we need to analyze
\[
P_c (\ubar{s},\ubar{s}') \defeq
\bigabs{\frac{1}{n^2b_n^k} \sum_{(\ubar{t},\ubar{t}')\in\mc{T}^c(\ubar{s},\ubar{s}')} \left( \E a_n^{b}(\ubar{t}) a_n^{b}(\ubar{t}') - \E a_n^{b}(\ubar{t}) \E a_n^{b}(\ubar{t}') \right)}.
\]


\noindent\underline{1. Case: $\ubar{s}$, $\ubar{s}'$ have only even edges.}\newline
[Claim: We always have regular convergence. Summable convergence holds if $1/(nb_n)$ is summable.]
Bounding each summand using condition (AAU1) and estimate \eqref{summandsimpleestimate} we find
\[
\forall\,(\ubar{t},\ubar{t}')\in\mc{T}^c(\ubar{s},\ubar{s}'):\,\abs{\E a_n^{b}(\ubar{t}) a_n^{b}(\ubar{t}') - \E a_n^{b}(\ubar{t}) \E a_n^{b}(\ubar{t}')} \leq B.
\]
Then Lemma~\ref{lem:doublehowmany} gives $P_c (\ubar{s},\ubar{s}') \leq k^2 \cdot k^{2k}\cdot B \cdot (n b_n)^{-1}$ and the claim follows.


\noindent\underline{2. Case: $\ubar{s}$ or $\ubar{s}'$ contains at least one odd edge}\newline
[Claim: We always have regular convergence. Summable convergence holds if $1/(nb_n)$ is summable.]
We start with the basic estimate
\[
P_c (\ubar{s},\ubar{s}') \leq \sum_{l=1}^k \frac{1}{n^2b_n^k}\sum_{(\ubar{t},\ubar{t}')\in\mc{T}^c_l(\ubar{s},\ubar{s}') }(\abs{\E a_n^{b}(\ubar{t}) a_n^{b}(\ubar{t}')} +\abs{\E a_n^{b}(\ubar{t}) } \cdot \abs{\E a_n^{b} (\ubar{t}')}).
\]
Thus, it suffices to show that each of the $k$ summands converges to zero, and that this convergence is summably fast provided $(n b_n)^{-1}$ is summable. To this end, pick an $l\in\{1,\ldots,k\}$ and a $(\ubar{t},\ubar{t}')\in\mc{T}^c_l (\ubar{s},\ubar{s}')$. We bound $\abs{\E a_n^{b}(\ubar{t})}\cdot\abs{\E a_n^{b}(\ubar{t}')}$ by condition (AAU1)
\[
\abs{\E a_n^{b}(\ubar{t})} \cdot \abs{\E a_n^{b} (\ubar{t}')} \leq \frac{
C_k^2
}{n^{\alpha\cdot (\kappa_1(\ubar{s}) +\kappa_1(\ubar{s}'))}} \leq \frac{
C_k^2
}{b_n^{\frac{1}{2}\cdot(\kappa_1(\ubar{s}) +\kappa_1(\ubar{s}'))}}.
\]
Now, to deal with $\abs{\E a_n^{b}(\ubar{t}) a_n^{b}(\ubar{t})}$ we first define a cycle $\ubar{u}\in\oneto{n}^{2k}$ according to Lemma~\ref{lem:graphsuperposition} which passes through the graph obtained by superposition of the graphs of $\ubar{t}$ and $\ubar{t}'$. Then we obtain
\[
\kappa_1(\ubar{u}) \geq \max(\kappa_1(\ubar{t})+\kappa_1(\ubar{t}')-2l,0) = \max(\kappa_1(\ubar{s})+\kappa_1(\ubar{s}')-2l,0)
\]
because for each common edge of $\ubar{t}$ and $\ubar{t}'$, the number of single edges can be reduced by at most $2$ after superposition of the graphs. We conclude
\[
\abs{\E a_n^{b}(\ubar{t}) a_n^{b}(\ubar{t})} = \abs{\E a_n^{b}(\ubar{u})} \leq \frac{
C_{2k}
}{n^{\alpha\cdot\kappa_1(\ubar{u})}} \leq \frac{
C_{2k}
}{b_n^{\frac{1}{2}\cdot\kappa_1(\ubar{u})}} \leq \frac{
C_{2k}
}{b_n^{\frac{1}{2}\cdot\max(\kappa_1(\ubar{s})+\kappa_1(\ubar{s}')-2l,0)}},
\]
Lemma~\ref{lem:doublehowmany} yields
$
\#\mc{T}^c_l(\ubar{s},\ubar{s}) \leq k^2 \cdot (2k)^{2k} \cdot nb_n^{K(\ubar{s}) + K(\ubar{s}')-l-1},
$
where $K(\ubar{s}) = \sum_{j}\kappa_j(\ubar{s}) \leq (k+\kappa_1(\ubar{s}))/2$ is defined as in \eqref{eq:case2numbersummands}. Recalling the definition of $B$ in \eqref{summandsimpleestimate} we obtain
\begin{eqnarray*}
P_c (\ubar{s},\ubar{s}')
&\leq& \frac{k^2 (2k)^{2k}}{n} b_n^{\frac{1}{2}(\kappa_1(\ubar{s})+\kappa_1(\ubar{s}'))-l-1} \left( \frac{
C_{2k}
}{b_n^{\frac{1}{2}\cdot\max(\kappa_1(\ubar{s})+\kappa_1(\ubar{s}')-2l,0)}} +  \frac{
C_k^2
}{b_n^{\frac{1}{2}\cdot(\kappa_1(\ubar{s}) +\kappa_1(\ubar{s}'))}}       \right)\\
&\leq& \frac{k^2 (2k)^{2k}}{n}\cdot
B \, \cdot
\frac{b_n^{\frac{1}{2}(\kappa_1(\ubar{s})+\kappa_1(\ubar{s}')-2l)-1}}{b_n^{\frac{1}{2}\cdot\max(\kappa_1(\ubar{s})+\kappa_1(\ubar{s}')-2l,0)}} \;
\leq \; 
B \, \cdot
\frac{k^2 (2k)^{2k}}{n b_n}
\end{eqnarray*}
and the claim follows. We have now gathered all the information that we need to prove statements of
i) through iv) of Theorem~\ref{thm:variancetozero}:

\begin{enumerate}[i)]
	\item As can be seen from the outcome of each of the cases, the condition that $b_n\to \infty$ suffices for regular convergence of the variance to zero.
	\item Assuming $\{+1,-1\}$-valued entries, we observe that the term in \eqref{eq:star} vanishes for each subsum in our case distinction except in "Step 1, Case 3" and "Step 2, Case 2" (since if $\ubar{t}$, say, consists of only even edges, we have $a_n(\ubar{t})=1$). In those cases we need for a summable convergence that $\frac{1}{b_n^3}$ is summable and that $\frac{1}{nb_n}$ is summable. Since the former implies the latter, it is enough to assume that $\frac{1}{b_n^3}$ is summable.
	\item Without extra assumptions on the entries of $(a_n)_n$, all of the above subcases are relevant. Therefore, for a summable convergence of the variance we need $\alpha>\frac{1}{2}$ and the sequences $(\frac{1}{b_n^2})_n$, $(\frac{1}{n}D^{(l)}_n)_n$ and $(C^{(l)}_n)_n$ for all $l\in\N$ to be summable over $n$. In particular, we use condition (AAU3).
	\item Assuming i.i.d. entries in $(a_n)_n$ with existing moments, zero expectation and unit variance, we see that
$\E a^{b}_n(\ubar{t}) a^{b}_n(\ubar{t}') - \E a^{b}_n(\ubar{t}) \E a^{b}_n(\ubar{t}')=0$ for all tuples $(\ubar{t},\ubar{t}')$ with disjoint edge sets. Therefore, only the partial sums of "Step 2" contribute in \eqref{eq:star}. The two cases there require summability of $(\frac{1}{nb_n})_n$ to conclude summable convergence.	
\end{enumerate}
\end{proof}

\sloppy
\printbibliography
\noindent\textsf{(Michael Fleermann and Werner Kirsch)\newline
FernUniversit\"at in Hagen\newline
Fakult\"at f\"ur Mathematik und Informatik\newline
58084 Hagen, Germany}\newline
\textit{E-mail address:}
\texttt{michael.fleermann@fernuni-hagen.de}\\\textit{E-mail address:}
\texttt{werner.kirsch@fernuni-hagen.de}
\vspace{5mm}

\noindent\textsf{(Thomas Kriecherbauer)\newline
Universität Bayreuth\newline
Mathematisches Institut\newline
95440 Bayreuth, Germany}\newline
\textit{E-mail address:}
\texttt{thomas.kriecherbauer@uni-bayreuth.de}

\end{document}